\documentclass[12pt]{article}
\usepackage{amsmath,amsfonts,amssymb,amsthm,mathrsfs,caption}
\usepackage{tikz}
\usepackage{pgfplots}

\newtheorem{proposition}{Proposition}[section]
\newtheorem{theorem}[proposition]{Theorem}
\newtheorem{corollary}[proposition]{Corollary}
\newtheorem{lemma}[proposition]{Lemma}
\newtheorem{remark}[proposition]{Remark}

\numberwithin{equation}{section}
\newcommand{\nc}{\newcommand}
\nc{\R}{{\mathbb R}}
\nc{\MM}{{\mathbb M}}
\nc{\bS}{{\mathbb S}^{d-1}}
\nc{\N}{{\mathbb N}}
\nc{\C}{{\mathbb C}}
\nc{\Z}{{\mathbb Z}}
\nc{\BP}{\mathbb{P}}
\nc{\BE}{\mathbb{E}}
\nc{\BQ}{\mathbb{Q}}
\nc{\bN}{{\mathbf N}}
\nc{\BS}{{\mathbb S}}
\nc{\BX}{{\mathbb X}}
\nc{\BY}{{\mathbb Y}}
\nc{\cB}{{\mathcal B}}
\nc{\cX}{{\mathcal X}}
\nc{\cY}{{\mathcal Y}}
\nc{\dint}{{\rm d}}
\nc{\D}{\Delta}
\nc{\g}{\gamma}
\nc{\cI}{\mathcal{I}}
\nc{\cZ}{\mathcal{Z}}
\nc{\cum}{\operatorname{cum}}
\nc{\sZ}{{\mathscr{Z}}}

\DeclareMathOperator{\BV}{\operatorname{Var}}

\DeclareMathOperator{\closure}{cl}

\DeclareMathOperator{\Vol}{Vol}
\DeclareMathOperator{\interior}{int}

\usepackage{color}

\makeatletter
\let\@fnsymbol\@alph
\makeatother

\begin{document}
\title{\bf Moderate deviations on Poisson chaos}
\date{}

\renewcommand{\thefootnote}{\fnsymbol{footnote}}

\author{Matthias Schulte\footnotemark[1]\; and Christoph Th\"ale\footnotemark[2]\,}

\footnotetext[1]{Institute of Mathematics, Hamburg University of Technology, Germany. E-mail: matthias.schulte@tuhh.de}

\footnotetext[2]{Faculty of Mathematics, Ruhr University Bochum, Germany. E-mail: christoph.thaele@rub.de}

\maketitle

\begin{abstract}
This paper deals with U-statistics of Poisson processes and multiple Wiener-It\^o integrals on the Poisson space. Via sharp bounds on the cumulants for both classes of random variables, moderate deviation principles, concentration inequalities and normal approximation bounds with Cram\'er correction are derived. It is argued that the results obtained in this way are in a sense best possible and cannot be improved systematically. Applications in stochastic geometry and to functionals of Ornstein-Uhlenbeck-L\'evy processes are investigated.

\end{abstract}

\section{Introduction}\label{sec:Introduction}

Probabilistic limit theorems for functionals of Poisson processes were intensively studied over the past decades. Starting with the seminal work \cite{PeccatiSoleTaqquUtzet}, this field of research got a particular new drive. In that paper, Malliavin calculus for Poisson processes was combined for the first time with Stein's method for normal approximation to deduce new central limit theorems with explicit error bounds. Previously, this connection had been established and exploited for functionals of Gaussian processes and has led to a large number of exciting new developments, see e.g.\ the monograph \cite{NP11Book} for an excellent introduction. As an example we mention the celebrated fourth moment theorem, which states that a sequence of random variables living inside a fixed Wiener chaos and having unit variance converges in distribution to a standard Gaussian random variable if and only if their fourth moments converge to $3$, the fourth moment of the standard Gaussian distribution. For the Poisson space, which we consider throughout this paper, a fourth moment theorem in the same spirit as well as some refinements were established in \cite{DoeblerPeccatiAoP18,DoeblerPeccatiECP18,DoeblerVidottoZheng}. Indeed, for a sequence of random variables $(F_n)_{n\in\N}$ living inside a fixed Poisson chaos, so-called multiple Wiener-It\^o integrals, such that $\BE[F_n^2]=1$ for each $n\in\N$ one has that
\begin{align*}
F_n\overset{d}{\longrightarrow}N\sim\mathcal{N}(0,1)\qquad\text{if}\qquad\BE[F_n^4]\to 3,
\end{align*}
as $n\to\infty$. This can equivalently be rephrased by saying that
\begin{align*}
F_n\overset{d}{\longrightarrow}N\sim\mathcal{N}(0,1)\qquad\text{if}\qquad\cum_4(F_n)\to 0,
\end{align*}
as $n\to\infty$, where $\cum_4(F_n):=\BE[F_n^4]-3$ stands for the fourth cumulant of $F_n$. In case that {the sequence} $(F_n^4)_{n\in\mathbb{N}}$ is uniformly integrable also the reverse direction {of this implication} is true.

Poisson functionals, {i.e.\ random variables depending only on a Poisson process,} play a crucial role in stochastic geometry, where one frequently studies random structures constructed from an underlying Poisson process. For such {situations} the Malliavin-Stein method is a very useful approach that has been extensively employed over the last years (see e.g.\ the volume of survey articles \cite{PeccatiReitznerBook}). Although multiple Wiener-It\^o integrals are very important and interesting objects, statistics of interest in stochastic geometry are usually not single multiple Wiener-It\^o integrals as considered in the fourth moment theorem above. {However,} many of them are so-called U-statistics of Poisson processes. Since these Poisson U-statistics can be written as finite sums of multiple Wiener-It\^o integrals, they are closely related to multiple Wiener-It\^o integrals and very well suited for the Malliavin-Stein approach. This technique was used in e.g.\ \cite{BourguinPeccati,EichelsbacherThaele,LachiezeReyPeccati1EJP13,LachiezeReyPeccati2SPA13,LPST,PeTh13,ReitznerSchulte,SchulteKolmogorov} to derive general normal approximation results for Poisson U-statistics, which were applied to different situations such as 
\begin{itemize}
\item[(i)] geometric random graphs \cite{BourguinPeccati,LachiezeReyPeccati1EJP13,LachiezeReyPeccati2SPA13,ReitznerSchulte,ReitznerSchulteThaele},
\item[(ii)] intersection and flat processes \cite{BetkenHugThaele,HeroldHugThaele,HTW,LPST,ReitznerSchulte,SchulteThaele14Flats},
\item[(iii)] {random} simplicial complexes \cite{AkinwandeReitzner,DecreusefondFerrazRandriambololonaVergne},
\item[(iv)] the statistical analysis of spherical Poisson fields \cite{BourguinDurastantiMarinucciPeccati,BourguinDurastantiMarinucciPeccati_survey}.
\end{itemize}  
We also point to the works \cite{BachmannPeccati,BachmannReitzner} for the study of concentration bounds for U-statistics of Poisson processes. For a survey on Poisson U-statistics we refer to \cite{LachiezeReyReitzner}. 

The present paper is focussed on refinements of the central limit theorem for {multiple Wiener-It\^o integrals and U-statistics of Poisson processes.} For finite sums {of such random variables we} study the validity of the Gaussian tail behaviour on scales beyond the one of the central limit theorem. We do this by proving moderate deviation principles (MDPs) as well as concentration inequalities and normal approximation bounds with Cram\'er correction. 

Our proofs rely on the so-called method of cumulants, which requires fine estimates on the cumulants of all orders. It is well known that such bounds encode much information about the fine probabilistic behaviour of the involved random variables, see the monograph \cite{SaulisBuch}. In particular, sharp bounds on cumulants lead to moderate deviation principles, see \cite{DoeringEichelsbacher}. For more details on the method of cumulants we {also} refer to the recent survey \cite{DJSsurvey}. In order to control the cumulants, so-called product formulas for the moments and cumulants of multiple Wiener-It\^o integrals are needed. In the present article we improve existing results in this direction (see e.g.\ \cite{LPST,PeTa,Surgailis}) by deriving such bounds under weaker (and partially even optimal) integrability assumptions. These findings are of independent interest.

Our work can be regarded as a continuation of our previous article \cite{SchulteThaele2016}, where we studied similar questions for multiple stochastic integrals on the Wiener space, that is, stochastic integrals with respect to Gaussian processes. On the Wiener space it has been shown in \cite{SchulteThaele2016} that all cumulants of a multiple stochastic integral are bounded in terms of the fourth cumulant. Roughly speaking, this can be seen as a consequence of the hypercontractivity property on the Wiener space. Since no such property is available for Poisson processes, a similar fourth-cumulant-phenomenon cannot be expected for the classes of random variables we consider. This together with the much more involved combinatorial structure of the product formulas for multiple stochastic integrals on the Poisson space makes the derivation of cumulant estimates for U-statistics and multiple stochastic integrals a challenging problem, which is tackled in the present text. It is one of the main features that our results will turn out to be best possible. In fact, we shall identify a range of scales on which general finite sums of Poisson U-statistics and general finite sums of multiple Wiener-It\^o integrals satisfy a MDP and we construct examples of such functionals that cannot satisfy a similar MDP beyond this range of scales. We highlight that this is in sharp contrast to the situation studied in \cite{SchulteThaele2016}, where such an example could not be found so far. This in turn led to a range of scalings for which we could not answer in \cite{SchulteThaele2016} whether or not a MDP is valid for a general sequence of multiple stochastic integrals.

Our general findings for multiple Wiener-It\^o integrals and Poisson U-statistics will be illustrated by means of three examples. We start by specialising our estimates to U-statistics having a fixed kernel. As an application we consider the intersection process of order $q$ generated by a Poisson process of $k$-dimensional totally geodesic submanifolds in a $d$-dimensional standard space of constant curvature $\kappa\in\{-1,0,1\}$. More specifically, we consider the $d-q(d-k)$-dimensional Riemannian volume associated with such an intersection process within a fixed observation window. This naturally connects to the recent line of research in non-Euclidean stochastic geometry. As a second model we investigate the random geometric graph in which two points of a homogeneous Poisson process within some convex body in $\R^d$ are connected by an edge, provided their Euclidean distance does not exceed some given threshold. The Poisson functional we consider is a linear combination of classical subgraph  counting statistics. Finally we study the Ornstein-Uhlenbeck process generated by a Poisson process in space and time. More precisely, our focus lies on the quadratic variation functional of this stochastic process, which admits a representation as a sum of Wiener-It\^o integrals of order one and two.

\bigskip

This paper is organised as follows. After introducing some preliminaries and notation in Section \ref{sec:preliminaries}, we present and discuss our main results in Section \ref{sec:MainResults}. We consider moderate deviations for multiple Wiener-It\^o integrals and Poisson U-statistics in Subsection \ref{subsec:ModerateDeviations+}, while Subsection \ref{subsec:ProductFormulas} deals with product formulas, which are essential ingredients of our proofs. Section \ref{sec:Applications} is devoted to applications, before the proofs are given in Sections \ref{sec:ProofProducFormula}, \ref{sec:ProofsMainResults} and \ref{sec:ProofApplications}.

\section{{Preliminaries}}\label{sec:CombinatoricsPoissonSpaces}\label{sec:preliminaries}

\subsection{Poisson processes and multiple Wiener-It\^o integrals}\label{subsec:PoissonSpacesAndMultipleIntegrals}

Let $(\BX,\cX)$ be a measurable space, which is supplied with a $\sigma$-finite measure $\mu$. A random counting measure $\eta$ on $\BX$ is called a Poisson process with intensity measure $\mu$, provided that
\begin{itemize}
\item[i)] for all $B\in\cX$, $\eta(B)$ is a (possibly {degenerate}) Poisson distributed random variable with mean $\mu(B)$,
\item[ii)] if $B_1,\ldots,B_n\in\cX$, $n\in\mathbb{N}$, are {pairwise disjoint,} the random variables $\eta(B_1),\ldots,\eta(B_n)$ are independent. 
\end{itemize}
For $q,r\in\N$ we let $L^r(\mu^q)$ be the space of measurable functions $f:\BX^q\to\R$ with the property that $|f|^r$ is integrable with respect to $\mu^q$, the $q$-fold product measure of the underlying measure $\mu$. Moreover, we shall denote by $L^r_s(\mu^q)$ the subspace of symmetric functions, that is, functions $f\in L^2(\mu^q)$ that are invariant with respect to arbitrary permutations of their arguments. 

For a counting measure $\xi$ on $\BX$ and $m\in\N$ let us define the measure $\xi^{(m)}$ on $\BX^m$ by
\begin{align*}
\xi^{(m)}(\,\cdot\,) := \int_{\BX}\cdots\int_{\BX} &{\bf 1}((x_1,\ldots,x_m)\in\,\cdot\,)\,\Big(\xi-\sum_{i=1}^{m-1}\delta_{x_i}\Big)(\dint x_m)\Big(\xi-\sum_{i=1}^{m-2}\delta_{x_i}\Big)(\dint x_{m-1})\cdots\\
&\times (\xi-\delta_{x_1})(\dint x_2)\xi(\dint x_1)\,,
\end{align*}
where $\delta_x$ stands for the Dirac measure at $x\in\BX$. For $f\in L_s^1(\mu^q)$ we define the pathwise multiple stochastic integral $I_q(f)$ of $f$ with respect to the (compensated) Poisson process $\eta$ by
$$
I_q(f) := \sum_{J\subset[q]}(-1)^{q-|J|}\int_{\BX^{|J|}}\int_{\BX^{q-|J|}}f(x_1,\ldots,x_q) \, \mu^{q-|J|}(\dint x_{J^c}) \, \eta^{(|J|)}(\dint x_J)
$$
with $[q]:=\{1,\ldots,q\}$, $x_J{:=}(x_j:j\in J)$ and where $|J|$ stands for the cardinality of $J$ (for {$J=[q]$} we interpret the inner integral as $f(x_1,\ldots,x_q)$). The $q$-fold Wiener-It\^o integral of $f\in L_s^2(\mu^q)$ is defined as {the limit of $(I_q(f_n))_{n\in\N}$ in the space of square integrable random variables, where $(f_n)_{n\in\N}$ is a sequence of simple symmetric functions approximating $f$ in $L_s^2(\mu^q)$. We} recall that $\BE[I_q(f)]=0$ and that 
$$
\BE[I_{q_1}(f_1)I_{q_2}(f_2)]=q_1!{\bf 1}(q_1=q_2)\,\langle f_1,f_2\rangle_{L^2(\mu^{q_1})},
$$
where $q_1,q_2\in\N$, $f_1\in L_s^2(\mu^{q_1})$, $f_2\in L_s^2(\mu^{q_2})$ and $\langle\,\cdot\,,\,\cdot\,\rangle_{L^2(\mu^{q_1})}$ denotes the usual scalar product in $L^2(\mu^{q_1})$. In particular, if a random variable has the form $F:=I_{q_1}(f_1)+\ldots+I_{q_k}(f_k)$ with $k\in\N$, distinct $q_1,\ldots,q_k\in\N$ and {$f_i\in L_s^2(\mu^{q_i})$}, $i\in\{1,\ldots,k\}$, we have that $\BE F=0$ and that the variance of $F$ is given by
\begin{align}\label{eq:VarianceIntegrals}
\BV F = q_1!\|f_{q_1}\|_{L^2(\mu^{q_1})}^2+\ldots+q_k!\|f_{q_k}\|_{L^2(\mu^{q_k})}^2\,.
\end{align}
Let us also remark that the collection of all random variables of the form $I_q(f)$ with {$f\in L^2_s(\mu^q)$} is called the $q^{\rm th}$ Poisson chaos (with respect to the measure $\mu$).

We refer to \cite{LastPenroseBook} for background material concerning Poisson processes and a detailed construction of the multiple Wiener-It\^o integral.

\subsection{Poisson U-statistics}\label{subsec:UStatistics}

As in the previous section, let $(\BX,\cX)$ be a measurable space and $\eta$ be a Poisson process on $\BX$ with $\sigma$-finite intensity measure $\mu$. A Poisson U-statistic is a random variable of the form 
$$
S:=\sum_{(x_1,\hdots,x_q)\in\eta^q_{\neq}} f(x_1,\hdots,x_q),
$$
where $q\in\N$, $\eta^q_{\neq}$ stands for the set of all $q$-tuples of distinct points of $\eta$ and $f\in L^1_s(\mu^q)$. Since we sum over all permutations of any combination of $q$ distinct points of $\eta$, we can assume without loss of generality that $f$ is symmetric. We denote $q$ as order and refer to $f$ as the kernel of $S$. The previous definition covers most of the relevant situations. However, our general framework even allows situations where the Poisson process is not given by its atoms, cf.\ \cite[Section 12.3]{LastPenroseBook}. In this case, we mean by a Poisson U-statistic of order $q$ a functional of the type
$$
S:=\int_{\BX^q} f(x_1,\hdots,x_q) \, \eta^{(q)}(\dint (x_1,\hdots,x_q))\,.
$$ 
In the sequel we always use the notation with the sum since we believe that it is more intuitive and is the typical situation for most of our examples. 

The Poisson U-statistic $S$ is square integrable if and only if
\begin{equation}\label{eqn:ConditionKernelUstatistic}
\int_{\BX^i} \bigg(\int_{\BX^{q-i}} f(y_1,\hdots,y_i,x_1,\hdots,x_{q-i}) \, \mu^{q-i}(\dint(x_1,\hdots,x_{q-i}))\bigg)^2 \, \mu^i(\dint(y_1,\hdots,y_i)) <\infty
\end{equation}
for all $i\in\{0,\hdots,q\}$ (see \cite[Proposition 12.12]{LastPenroseBook} and \cite[Section 3]{ReitznerSchulte}). For $i=0$ and $i=q$ this means that {$f\in L_s^1(\mu^q)$ and $f\in L_s^2(\mu^q)$,} respectively. If \eqref{eqn:ConditionKernelUstatistic} is satisfied, the functions $f_i:\BX^i\to\R$, $i\in\{1,\hdots,q\}$, given by
\begin{equation}\label{eqn:kernels}
f_i(y_1,\hdots,y_i):=\binom{q}{i} \int_{\BX^{q-i}} f(y_1,\hdots,y_i,x_1,\hdots,x_{q-i}) \, \mu^{q-i}(\dint(x_1,\hdots,x_{q-i})),
\end{equation}
are square integrable. {In \cite[Section 3]{ReitznerSchulte} it is shown that a square integrable Poisson U-statistic} $S$ has the representation
\begin{equation}\label{eqn:ChaosExpansionUstatistic}
S=\BE S + \sum_{i=1}^q I_i(f_i),
\end{equation}
and {that} its variance is given by
\begin{equation}\label{eqn:VarianceUstatistic}
\BV S = \sum_{i=1}^q i! \|f_i\|^2_{L^2(\mu^i)}\,,
\end{equation}
compare with \eqref{eq:VarianceIntegrals}. The decomposition \eqref{eqn:ChaosExpansionUstatistic} is called the Wiener-It\^o chaos expansion of $S$. We {emphasise} that any square-integrable Poisson functional $F$, i.e.\ any random variable depending on a Poisson process only, has a representation as a sum of its expectation and (possibly infinitely many) multiple Wiener-It\^o integrals. In other words, \eqref{eqn:ChaosExpansionUstatistic} says that square-integrable Poisson U-statistics have a finite Wiener-It\^o chaos expansion. On the other hand, in \cite[Theorem {3.6}]{ReitznerSchulte} it is shown that any square-integrable Poisson functional with a finite Wiener-It\^o chaos expansion, whose kernels $f_i$ are {integrable,} can be written as a sum of finitely many Poisson U-statistics and a constant.  

\subsection{The method of cumulants}\label{subsec:CumulantsApproach}

In this section we present what is called the method of cumulants. We start with the definition of cumulants and by setting up the notation. For real-valued random variables $X_1,\hdots,X_m$, $m\in\N$, the joint characteristic function $\varphi_{X_1,\hdots,X_m}: \R^m\to\C$ is given by
$$
\varphi_{X_1,\hdots,X_m}(t_1,\hdots,t_m):=\BE \exp\bigg({\bf i}\sum_{i=1}^m t_i X_i \bigg)\,,
$$
where ${\bf i}$ is the imaginary unit. The joint cumulant of $X_1,\hdots,X_m$ is then defined as
$$
\cum(X_1,\hdots,X_m):=(-{\bf i})^m\frac{\partial^m}{\partial t_1\hdots \partial t_m}\log\varphi_{X_1,\hdots,X_m}(t_1,\hdots,t_m)\Big|_{t_1=\hdots=t_m=0}\,.
$$
In the following we consider random variables that have finite moments of all orders. This implies that all {joint} cumulants of these random variables are well-defined. Note that the joint cumulants are linear in each coordinate. For a real-valued random variable $X$ and $m\in\N$ we shall write $\cum_m(X){:=}\cum(X,\hdots,X)$ for the $m^{\rm th}$ cumulant of $X$.

Before we can {summarise} the main elements of the method of cumulants, we provide the definition of a moderate deviation principle. Let us recall from \cite{DemboZeitouni} that a sequence $(\BP_n)_{n\in\N}$ of probability measures on a topological space $\cZ$ with $\sigma$-field $\sZ$ satisfies a large deviation principle with speed $s_n\to\infty$ and good rate function {$\cI$} if the level sets $\{z {\in\cZ}:\cI(z)\leq a\}$ are compact for all $0\leq a<\infty$ and if for all $A\in\sZ$,
\begin{equation*}
	-\inf_{z\in\interior(A)}\cI(z)\leq \liminf_{n\to\infty} s_n^{-1}\log\BP_n(A) \leq \limsup_{n\to\infty}s_n^{-1}\log\BP_n(A)\leq -\inf_{z\in\closure(A)}\cI(z)\,,
\end{equation*}
where $\interior(A)$ and $\closure(A)$ stand for the interior and the closure of $A$, respectively. Moreover, a sequence $(X_n)_{n\in\N}$ of random variables satisfies a LDP if their distributions do. We will speak about a moderate deviation principle (MDP) instead of a LDP if the scaling of the involved random variables is between that of a law of large numbers and that of a central limit theorem.

Now, let $(X_n)_{n\in\N}$ be a sequence of square-integrable random variables, $\gamma \geq 0$ be a constant and $(\Delta_n)_{n\in\N}$ be a positive real-valued sequence. To keep the presentation of our results more transparent, we introduce the following shorthand notation and say that $(X_n)_{n\in\N}$ satisfies
\begin{itemize}
	\item $\boldsymbol{{\rm MDP}}(\gamma, (\Delta_n)_{n\in\N})$ if for any {positive real-valued} sequence $(a_n)_{n\in\N}$ with
	$$
	\lim\limits_{n\to\infty} a_n=\infty \quad \text{ and } \quad \lim\limits_{n\to\infty} \frac{a_n}{ \Delta_n^{1/(1+2\gamma)}}=0
	$$  
	the re-scaled random variables $(a_n^{-1}(X_n-\BE X_n)/\sqrt{\BV X_n})_{n\in\N}$ satisfy a moderate deviation principle (MDP) with speed $a_n^2$ and good rate function $\cI(z)=z^2/2${,}
	\item $\boldsymbol{{\rm CI}}(\gamma, (\Delta_n)_{n\in\N})$ if the {Bernstein-type} concentration inequality
	$$
	\BP(|X_n-\BE X_n|\geq z \sqrt{\BV X_n})\leq 2\exp\left(-\frac{1}{4}\min\Big\{\frac{z^2}{2^{1+\gamma}},\,(z\D_n)^{1/(1+\g)}\Big\}\right)
	$$
	holds for all {$n\in\N$ and $z\geq 0$,} 
	\item $\boldsymbol{{\rm NACC}}(\gamma, (\Delta_n)_{n\in\N})$ {if} a normal approximation bound with Cram\'er correction holds, that is, if
	 there exist constants $c_0,c_1,c_2>0$ only depending on $\gamma$ such that for all {$n\in\N$ and $z\in[0,c_0\Delta_n^{1/(1+2\gamma)}]$,}
	$$
	\BP(X_n-\BE X_n\geq z \sqrt{\BV X_n}) = {e^{L_{n,z}^{+}}(1-\Phi(z))\Big(1+c_1\theta^+_{n,z}\frac{1+z}{\Delta_n^{1/(1+2\gamma)}}\Big)}
	$$
	and
	$$
	{\BP(X_n-\BE X_n\leq -z \sqrt{\BV X_n}) = e^{L_{n,z}^{-}}(1-\Phi(z))\Big(1+c_1\theta^-_{n,z}\frac{1+z}{\Delta_n^{1/(1+2\gamma)}}\Big)}
	$$
	{with $\theta_{n,z}^+,\theta_{n,z}^-\in[-1,1]$ and $L_{n,z}^+,L_{n,z}^-\in (-c_2z^3/\Delta_n^{1/(1+2\gamma)},c_2z^3/\Delta_n^{1/(1+2\gamma)})$, where $\Phi$} is the distribution function of a standard Gaussian random variable.
\end{itemize}
The main tool for proving our results are sharp estimates for cumulants and their implications. The next proposition is our main device. {It collects findings taken from the monograph \cite{SaulisBuch}, the paper \cite{DoeringEichelsbacher} and the survey article \cite{DJSsurvey} (see also \cite[Lemma 11]{SchulteThaele2016}).} It {summarises} fine probabilistic estimates, which are available under certain natural bounds on cumulants.

\begin{proposition}[MDP, CI and NACC under cumulant bounds]\label{prop:CumulantBoundImplyMDPandConcentration}
	Let $(X_n)_{n\in\N}$ be a sequence of real-valued random variables such that $\BE[X_n]=0$, $\BE[X_n^2]=1$ and $\BE[|X_n|^m]<\infty$ for all $m\geq 3$. Suppose that there exist a constant $\gamma\geq 0$ and a positive real-valued sequence $(\Delta_n)_{n\in\N}$ such that
	\begin{equation}\label{eq:CumumantBound}
		|\cum_m(X_n)|\leq \frac{(m!)^{1+\gamma}}{\Delta_n^{m-2}}
	\end{equation}
for all $m\geq 3$ and $n\in\N$. Then $(X_n)_{n\in\N}$ satisfies $\boldsymbol{{\rm MDP}}(\gamma,(\Delta_n)_{n\in\N})$, $\boldsymbol{{\rm CI}}(\gamma,(\Delta_n)_{n\in\N})$ and $\boldsymbol{{\rm NACC}}(\gamma,(\Delta_n)_{n\in\N})$.
\end{proposition}

\begin{remark}\rm 
The cumulant bound \eqref{eq:CumumantBound} immediately implies a central limit theorem for the random variables $(X_n)_{n\in\mathbb{N}}$, as soon as $\Delta_n\to\infty$, if we let $n\to\infty$. In addition, it also delivers a bound for the speed of convergence in terms of the Kolmogorov distance defined as the sup-norm of the difference of the distribution function of $X_n$ and that of a standard Gaussian random variable, see e.g.\ \cite[Corollary 2.1]{SaulisBuch}. However, since this leads for all the random variables we consider in this paper to rates that are weaker than those already available in the existing literature, we have decided not to pursue this direction in this text.
\end{remark}

\subsection{Partitions}\label{subsec:Partitions}

Let $m\in\N$ and let $q_1,\hdots,q_m\in\N$. We define $N_0:=0$, $N_\ell{:=}\sum_{i=1}^\ell q_i$, $\ell\in\{1,\hdots,m\}$, and $N:=N_m$, and put $J_\ell:=\{N_{\ell-1}+1,\hdots,N_\ell\}$, $\ell\in\{1,\hdots,m\}$. 
A partition $\sigma$ of $[N]:=\{1,\hdots, N\}$ is a collection $\{B_1,\hdots,B_k\}$ of {\color{green} } $1\leq k\leq N$ {pairwise} disjoint non-empty sets, called blocks, such that $B_1\cup\ldots\cup B_k=[N]$. The number $k$ of blocks of $\sigma$ is denoted by $|\sigma|$. By $\Pi(q_1,\hdots,q_m)$ we denote the set of all partitions $\sigma$ such that $|B\cap J_\ell|\leq 1$ for all $\ell\in\{1,\hdots,m\}$ and $B\in\sigma$.

\begin{minipage}{0.44\columnwidth}
	\centering
	\includegraphics[width=0.5\columnwidth]{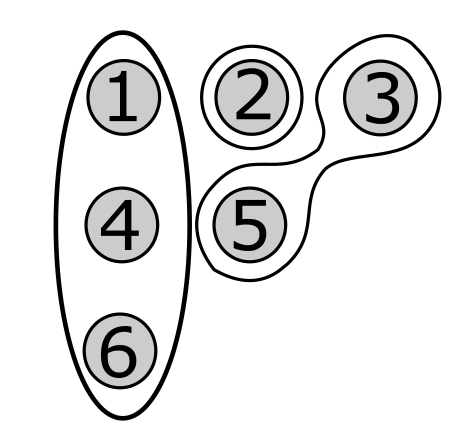}
	\captionof{figure}{\small An element of $\Pi(3,2,1)$.}
	\label{fig:1}
\end{minipage}
\begin{minipage}{0.56\columnwidth}
It is instructive to graphically represent a partition $\sigma\in\Pi(q_1,\ldots,q_m)$ as follows. We imagine the $q_1+\ldots+q_m$ elements of $\{1,\ldots,q_1+\ldots {+}q_m\}$ arranged in an array of $m$ rows, where the numbers $1,\ldots,q_1$ {(i.e.\ the elements of $J_1$)} form the first row, $q_1+1,\ldots,{q_1+}q_2$ {(i.e.\ the elements of $J_2$)} the second row and so on. The blocks of the partition $\sigma$ are then indicated by closed curves, where all elements encircled by the same curve belong to the same block of $\sigma$, see Figure \ref{fig:1}.
\end{minipage}

\medspace

Every partition $\sigma\in\Pi(q_1,\hdots,q_m)$ induces a partition $\sigma^*$ of $\{1,\hdots,m\}$ in the following way: $i,j\in\{1,\hdots,m\}$ are in the same block of $\sigma^*$ whenever there is a block $B\in\sigma$ such that $|B\cap J_i|=1$ and $|B\cap J_j|=1$. Let $\widetilde{\Pi}(q_1,\hdots,q_m)$ be the set of all partitions $\sigma\in\Pi(q_1,\hdots,q_m)$ such that $|\sigma^*|=1$. By $\Pi_{\geq 2}(q_1,\hdots,q_m)$ and
$\widetilde{\Pi}_{\geq 2}(q_1,\hdots,q_m)$ we denote the sets of all $\sigma\in\Pi(q_1,\hdots,q_m)$ and of all $\sigma\in\widetilde{\Pi}(q_1,\hdots,q_m)$ such that $|B|\geq 2$ for all $B\in\sigma$. Finally, we introduce the set $\overline{\Pi}(q_1,\ldots,q_m)$ of all partitions $\sigma\in\Pi(q_1,\ldots,q_m)$ such that for each $\ell\in\{1,\hdots,m\}$ there exists a block $B\in\sigma$ with $|B|\geq 2$ and $B\cap J_\ell \neq\varnothing$. In other words, in each row in the graphical representation of $\sigma$ there exists at least one element, which belongs to some block $B\in\sigma$ with $|B|\geq 2$. In case that $q_1=\hdots=q_m= {q}$ we write $\Pi^m(q)$, $\widetilde{\Pi}^m(q)$, $\widetilde{\Pi}^m_{\geq 2}(q)$ and $\overline{\Pi}^m(q)$ instead of $\Pi(q_1,\hdots,q_m)$, $\widetilde{\Pi}(q_1,\hdots,q_m)$,  $\widetilde{\Pi}_{\geq 2}(q_1,\hdots,q_m)$ and $\overline{\Pi}(q_1,\hdots,q_m)$, respectively.

For functions $f^{(\ell)}: \BX^{q_\ell}\to\R$, $\ell\in\{1,\hdots,m\}$, we define their tensor product $\otimes_{\ell=1}^m f^{(\ell)}: \BX^N\to\R$ by
$$
(\otimes_{\ell=1}^m f^{(\ell)})(x_1,\hdots,x_N):=\prod_{\ell=1}^m f^{(\ell)}(x_{N_{\ell-1}+1},\hdots,x_{N_\ell})\,.
$$
For $\sigma\in\Pi(q_1,\hdots,q_m)$ the function $(\otimes_{\ell=1}^m f^{(\ell)})_\sigma: \BX^{|\sigma|}\to\R$ is obtained by replacing in $(\otimes_{\ell=1}^m f^{(\ell)})$ all variables that belong to the same block of $\sigma$ by a new common variable. Note that this way $(\otimes_{\ell=1}^m f^{(\ell)})_\sigma$ is only defined up to permutations of its arguments. Since in what follows we always integrate with respect to all arguments of $(\otimes_{\ell=1}^m f^{(\ell)})_\sigma$, this does not cause problems.

\section{Main results}\label{sec:MainResults}

\subsection{Moderate deviation estimates}\label{subsec:ModerateDeviations+}

We are now prepared to present {the} main results of this paper. They show that finite sums of multiple stochastic integrals on the Poisson space as well as finite sums of Poisson U-statistics satisfy ${\rm\bf MDP}(\gamma,(\Delta_n)_{n\in\N})$, ${\rm\bf CI}(\gamma,(\Delta_n)_{n\in\N})$ as well as ${\rm\bf NACC}(\gamma,(\Delta_n)_{n\in\N})$, and we determine the parameters $\gamma$ and $(\Delta_n)_{n\in\N}$ in both situations. In fact, it will turn out later that the parameter $\gamma$ we obtain cannot be improved systematically. In the next two theorems we implicitly assume that all occurring integrals are well defined. We start with the result for multiple Wiener-It\^o integrals.

\begin{theorem}[MDP, CI and NACC for multiple integrals]\label{thm:MDP}
Let $(\eta_n)_{n\in\N}$ be a family of Poisson processes over $\sigma$-finite measure spaces $((\BX_n,\cX_n,\mu_n))_{n\in\N}$ and let $f_n^{(i)}\in L^2_s(\mu_n^{q_i})$, $i\in\{1,\hdots,k\}$, $n\in\N$, with distinct $q_1,\hdots,q_k\in\N$ and $k\in\N$ be such that $\sum_{i=1}^k \|f_n^{(i)}\|^2_{L^2(\mu_n^{q_i})} >0$ for all $n\in\N$. Define $q:=\max\{q_1,\hdots,q_k\}$ {and} let $Y_n:=\sum_{i=1}^k I_{q_i}(f_n^{(i)})$ {for $n\in\N$.} Assume that there is a positive real-valued sequence $(\alpha_n)_{n\in\N}$ such that, for any $n\in\N$,
\begin{equation}\label{eq:ConditionIntegrals}
(\BV Y_n)^{-m/2}\bigg|\int_{\BX_n^{|\sigma|}} (\otimes_{\ell=1}^m f_n^{(i_\ell)} )_\sigma \, \dint\mu_n^{|\sigma|}\bigg| \leq \alpha_n^{m-2}
\end{equation}
for all $\sigma\in\widetilde{\Pi}_{\geq 2}(q_{i_1},\hdots,q_{i_m})$, $i_1,\hdots,i_m\in\{1,\hdots,k\}$ and $m\geq 3$. Then $(Y_n)_{n\in\N}$ satisfies $\boldsymbol{{\rm MDP}}(q-1, ({c_{q,k}}\alpha_n^{-1})_{n\in\N})$, $\boldsymbol{{\rm CI}}(q-1, ({c_{q,k}}\alpha_n^{-1})_{n\in\N})$ and $\boldsymbol{{\rm NACC}}(q-1, ({c_{q,k}}\alpha_n^{-1})_{n\in\N})$ with $c_{q,k}:=1/(kq^q)^3$.
\end{theorem}

We can compare the assumptions of Theorem \ref{thm:MDP} with {those} of the corresponding result for multiple Wiener-It\^o integrals on the Wiener space proved in \cite{SchulteThaele2016}. In the latter case, it was sufficient to impose a condition on the \textit{fourth} cumulant only in order to bound all higher-order cumulants and to deduce {\bf MDP}, {\bf CI} and {\bf NACC}. In contrast, our condition \eqref{eq:ConditionIntegrals} (and also the condition \eqref{eq:ConditionUstatistics} in case of Poisson U-statistics below) impose restrictions to {\it all} cumulants of order $m\geq 3$ simultaneously. In view of the complicated combinatorial structure and in view of the absent hypercontractivity property on the Poisson space, this is unavoidable by our method, which is based on sharp cumulant bounds.

\medskip

Next, we shall discuss a version of Theorem \ref{thm:MDP} for sums of Poisson U-statistics.

\begin{theorem}[MDP, CI and NACC for Poisson U-statistics]\label{thm:MDP2}
Let $(\eta_n)_{n\in\N}$ be a family of Poisson processes over $\sigma$-finite measure spaces $((\BX_n,\cX_n,\mu_n))_{n\in\N}$ and let $f_n^{(i)}: \BX_n^{q_i}\to \R $, $i\in\{1,\hdots,k\}$, $n\in\N$, with distinct $q_1,\hdots,q_k\in\N$ and $k\in\N$, be measurable and satisfy \eqref{eqn:ConditionKernelUstatistic} and $\sum_{i=1}^k \|f_n^{(i)}\|^2_{L^2(\mu_n^{q_i})} >0$ for all $n\in\N$. Define $q:=\max\{q_1,\hdots,q_k\}$ and the random variables
$$
Z_n:=\sum_{i=1}^k S_n^{(i)}\qquad\text{with}\qquad S_n^{(i)}:=\sum_{(x_1,\hdots,x_{q_i})\in\eta^{q_i}_{n,\neq}} f_n^{(i)}(x_1,\hdots,x_{q_i}), \quad i\in\{1,\hdots,k\},
$$
for $n\in\N$. Assume that there is a positive real-valued sequence $(\beta_n)_{n\in\N}$ such that, for any $n\in\N$,
\begin{equation}\label{eq:ConditionUstatistics}
(\BV Z_n)^{-m/2} \bigg|\int_{\BX_n^{|\sigma|}} (\otimes_{\ell=1}^m f_n^{(i_\ell)} )_\sigma \, \dint\mu_n^{|\sigma|}\bigg| \leq \beta_n^{m-2}
\end{equation}
for all $\sigma\in\widetilde{\Pi}(q_{i_1},\hdots,q_{i_m})$, $i_1,\hdots,i_m\in\{1,\hdots,k\}$ and $m\geq 3$. Then $(Z_n)_{n\in\N}$ satisfies $\boldsymbol{{\rm MDP}}(q-1, ({c_{q,k}}\beta_n^{-1})_{n\in\N})$, $\boldsymbol{{\rm CI}}(q-1, ({c_{q,k}}\beta_n^{-1})_{n\in\N})$ and $\boldsymbol{{\rm NACC}}(q-1, ({c_{q,k}}\beta_n^{-1})_{n\in\N})$ with $c_{q,k}:=1/(kq^q)^3$.
\end{theorem}

\bigskip

\begin{remark}\rm
{The assumption $\sum_{i=1}^k \|f_n^{(i)}\|^2_{L^2(\mu_n^{q_i})} >0$ for all $n\in\N$ in Theorems \ref{thm:MDP} and \ref{thm:MDP2} ensures that, for all $n\in\N$, $\BV Y_n>0$ and $\BV Z_n>0$, respectively. Indeed, by \eqref{eq:VarianceIntegrals}, we have
$$
\BV Y_n = \sum_{i=1}^k q_i! \|f_n^{(i)}\|^2_{L^2(\mu_n^{q_i})} \geq \sum_{i=1}^k \|f_n^{(i)}\|^2_{L^2(\mu_n^{q_i})} >0.
$$
In case of $Z_n$ there exists a unique $i_n\in\{1,\hdots,k\}$ such that $q_{i_n}=\max\{ q_1,\hdots,q_k: \|f_n^{(i)}\|^2_{L^2(\mu_n^{q_i})} >0 \}$. From \eqref{eqn:ChaosExpansionUstatistic} we deduce that the $q_{i_n}$-th kernel of the chaos expansion of $Z_n$ is $f_n^{(i_n)}$, whence, by \eqref{eq:VarianceIntegrals},
$
\BV Z_n \geq q_{i_n}! \|f_n^{(i_n)}\|^2_{L^2(\mu_n^{q_{i_n}})}>0.
$}
\end{remark}

\begin{remark}\rm 
Moderate deviation principles and also concentration inequalities are well known in the case of `classical' U-statistics based on {fixed numbers} of i.i.d.\ random variables, see e.g.\ \cite{EichelsbacherSchmock} and \cite{MajorBook}. Although the classical U-statistics and the Poisson U-statistics we consider are close in the $L^2$-sense by a result of Dynkin and Mandelbaum \cite{DynkinMandelbaum}, this is not sufficient to push the classical results to the Poisson case since we investigate U-statistics on an exponential scale. In addition we remark that -- in sharp contrast to typical assumptions imposed on classical U-statistics -- the random variables $Y_n$ and $Z_n$ considered in Theorem \ref{thm:MDP} and \ref{thm:MDP2} satisfy $\mathbb{E}[e^{sY_n}]=\mathbb{E}[e^{sZ_n}]=\infty$ for any $s>0$, provided that $f_n^{(i)}\geq 0$ for all $i\in\{1,\hdots,k\}$ and that for some ${i}\in\{1,\ldots,k\}$, $q_i\geq 2$ and $\|f_n^{(i)}\|_{L^2(\mu_n^{q_i})}>0$ (see \cite[Corollary 2]{LPST}). 
\end{remark}

Theorem \ref{thm:MDP} and Theorem \ref{thm:MDP2} imply a moderate deviation principle for a range of scales $(a_n)_{n\in\N}$ `close' to that in the related central limit theorem, which in turn would correspond to the choice $a_n=1$ for all $n\in\N$. We shall now discuss whether or not it is possible in general to enlarge this range of scales. To this end we use a simple example based on a sum of independent and identically distributed random variables. Let us denote for $q\in\{0,1,2,\ldots\}$ by {$H_q$} the $q^{\rm th}$ Poisson-Charlier polynomial for the Poisson distribution with parameter $1$. This family of orthogonal polynomials is recursively defined by
$$
H_0(x):=1\qquad\text{and}\qquad H_{q+1}(x):=xH_q(x-1)-H_q(x)\,,\qquad x\in\R,
$$
for integers $q\geq 0$, see e.g.\ \cite[Equation (10.0.2)]{PeTa}. For example,
\begin{align*}
H_1(x) &= x-1,\\
H_2(x) &= x^2-3x+1,\\
H_3(x) &= x^3-6x^2+8x-1,\\
H_4(x) &= x^4-10x^3+29x^2-24x+1.
\end{align*}
In particular, {$H_q$} is a polynomial of degree $q$ with leading coefficient equal to $1$.

\begin{theorem}\label{thm:CharlierPolynomials}
Fix $q\in\N$ and let $(Z_k)_{k\in\N}$ be a sequence of independent and Poisson distributed random variables with parameter $1$. For each $n\in\N$ define $S_n:=\sum_{k=1}^nH_q(Z_k)$ and let $(a_n)_{n\in\N}$ be a sequence of positive real numbers such that $\lim_{n\to\infty}a_n=\infty$ and $\lim_{n\to\infty}a_n/\sqrt{n}=0$. 
\begin{itemize}
\item[a)] Assume that
$$
\lim_{n\to\infty}a_n^{-2+1/q}n^{1/(2q)}\log(a_n\sqrt{n})=\infty.
$$
Then the sequence of random variables $\big(\frac{S_n}{a_n\,\sqrt{nq!}}\big)_{n\in\N}$ satisfies a MDP with speed $a_n^2$ and good rate function $\mathcal{I}(z)=z^2/2$.
\item[b)] Assume that
$$
\lim_{n\to\infty}a_n^{-2+1/q}n^{1/(2q)}\log(a_n\sqrt{n})<\infty.
$$
Then, the sequence of random variables $\big(\frac{S_n}{a_n\,\sqrt{nq!}}\big)_{n\in\N}$ cannot satisfy a MDP with a good rate function $\mathcal{I}(z)$ satisfying $\mathcal{I}(z)>0$ for $z\neq 0$ and with $\mathcal{I}(z)\to\infty$, as $z\to\pm\infty$.
\end{itemize}
\end{theorem}

To discuss the relation between Theorem \ref{thm:CharlierPolynomials}, which is not derived by the method of cumulants but by findings from \cite{Arcones,EichelsbacherLoewe} for sums of i.i.d.\ random variables, and Theorem \ref{thm:MDP}, we let $\eta$ be a Poisson process over a $\sigma$-finite measure space $(\BX,\cX,\mu)$ and $Z_1$ be a Poisson random variable with parameter $1$. Then, for every $q\in\N$ and every $B\in\cX$ with $\mu(B)=1$ we have that the random variable $H_q(Z_1)$ and the multiple Wiener-It\^o integral $I_q(g_B)$ with $g_B(z_1,\ldots,z_q):={\bf 1}_B^{\otimes q}(z_1,\ldots,z_q):={\bf 1}_B(z_1)\cdots{\bf 1}_B(z_q)$ are identically distributed, see \cite[Proposition 10.0.2]{PeTa}. Thus, for each $n\in\mathbb{N}$, $S_n$ has the same distribution as $I_q(f_n)$ with $f_n:=\sum_{i=1}^n g_{B_i}$ with pairwise disjoint measurable subsets $(B_i)_{i\in\mathbb{N}}$ of $\mathbb{X}$ with $\mu(B_i)=1$ for all $i\in\mathbb{N}$. For $m\in\mathbb{N}$ with $m\geq 3$ and $\sigma\in\widetilde{\Pi}^{m}_{\geq 2}(q)$, we obtain
$$
\int_{\BX^{|\sigma|}}(\otimes_{\ell=1}^m f_n)_\sigma\,\dint\mu^{|\sigma|} = \sum_{i_1,\hdots,i_m=1}^n \int_{\BX^{|\sigma|}}(\otimes_{\ell=1}^m g_{B_{i_\ell}})_\sigma\,\dint\mu^{|\sigma|} = \sum_{i=1}^n \int_{\BX^{|\sigma|}}(\otimes_{\ell=1}^m g_{B_{i}})_\sigma\,\dint\mu^{|\sigma|}=n,
$$
where we used the construction of the $(g_{B_i})_{i\in\mathbb{N}}$, the pairwise disjointness of $(B_i)_{i\in\mathbb{N}}$ and $\mu(B_i)=1$ for $i\in\mathbb{N}$. Consequently, the random variables $(S_n/\sqrt{n q!})_{n\in\mathbb{N}}$ satisfy \eqref{eq:ConditionIntegrals} with $\alpha_n:=\frac{1}{\sqrt{n}}$. So, Theorem \ref{thm:MDP} implies that if $(a_n)_{n\in\N}$ is a sequence of positive real numbers such that
\begin{equation}\label{eqn:Condition_Charlier}
\lim_{n\to\infty}a_n=\infty\qquad\text{and}\qquad\lim_{n\to\infty}{a_n\over\alpha_n^{-1/(2q-1)}}=\lim_{n\to\infty}{a_n\over n^{1/(4q-2)}}=0,
\end{equation}
then the sequence of random variables $(S_n/(a_n\sqrt{nq!}))_{n\in\mathbb{N}}$ satisfies a MDP with speed $a_n^2$ and good rate function $\mathcal{I}(z)=z^2/2$. However, the growth condition on $a_n$ we just obtained by means of the method of cumulants coincides up to subpolynomial factors with the optimal one from Theorem \ref{thm:CharlierPolynomials}, which in turn is based on a different method. Indeed, it is easy to verify that the condition of a) in Theorem \ref{thm:CharlierPolynomials} is satisfied in case of \eqref{eqn:Condition_Charlier}, while the condition of b) holds if
$$
\lim_{n\to\infty}{a_n\over n^{\varepsilon+1/(4q-2)}}=\infty
$$
for some $\varepsilon>0$. In other words this means that the polynomial order of the range of scalings in Theorem \ref{thm:MDP} and, thus, presumably also in Theorem \ref{thm:MDP2} cannot be improved systematically. On the other hand, this does not necessarily exclude the possibility that for special choices of functions $f_n$ the MDP might hold beyond this range of scales.

We also note in this context that such an optimality result is not available for sequences of multiple stochastic integrals on a Wiener space. In fact, it has been argued in \cite{SchulteThaele2016} that there exists a non-trivial interval of scales $a_n$ for which it is not clear whether or not a moderate deviation principle is satisfied in general. We find it rather remarkable that despite the much more involved combinatorial nature of stochastic integrals on Poisson spaces, such a gap does not exist in this set-up.

\subsection{Product formulas}\label{subsec:ProductFormulas}

One of the main devices for deriving sharp bounds on cumulants of multiple stochastic integrals and Poisson U-statistics are explicit combinatorial formulas for the (joint) cumulants of such random variables. The moment and cumulant formulas provided in this section are known, but we were able to derive them under weaker - sometimes even minimal - integrability assumptions in comparison with the existing literature. We consider the same framework as in Section \ref{sec:preliminaries}.

\begin{theorem}[Moment and cumulant formulas for stochastic integrals]\label{thm:CumulantFormulaMultipleIntegrals}
	Let $m\in\N$, let $\widetilde{m}\in\mathbb{N}$ with $\widetilde{m}\geq m$ be even and let $f^{(\ell)}\in L^2_s(\mu^{q_\ell})$ with $q_\ell\in\N$, $\ell\in\{1,\hdots,m\}$, be such that
	\begin{align}
	\int_{\BX^{|\sigma|}}|(\otimes_{\ell=1}^{\widetilde{m}}f^{(i)})_\sigma|\,\dint\mu^{|\sigma|} &<\infty,\qquad \sigma\in\Pi_{\geq 2}^{\widetilde{m}}(q_i),\,i\in\{1,\ldots,m\},\label{eq:CumThmAss1}\\
	\int_{\BX^{|\sigma|}} \big|(\otimes_{\ell=1}^m f^{(\ell)})_\sigma \big| \, \dint\mu^{|\sigma|} &<\infty,\qquad \sigma \in {\Pi}_{\geq 2}(q_1,\hdots,q_m)\,.\label{eq:CumThmAss2}
	\end{align}
	Then, 
	\begin{align*}
	\BE\Big[\prod_{\ell=1}^mI_{q_\ell}(f^{(\ell)})\Big] &=\sum_{\sigma\in{\Pi}_{\geq 2}(q_1,\hdots,q_m)} \int_{\BX^{|\sigma|}} (\otimes_{\ell=1}^m f^{(\ell)})_\sigma \, \dint\mu^{|\sigma|}\,,\\
	\cum(I_{q_1}({f^{(1)}}),\hdots,I_{q_m}({f^{(m)}}))&=\sum_{\sigma\in\widetilde{\Pi}_{\geq 2}(q_1,\hdots,q_m)} \int_{\BX^{|\sigma|}} (\otimes_{\ell=1}^m f^{(\ell)})_\sigma \, \dint\mu^{|\sigma|}\,.
	\end{align*}
\end{theorem}

Formulas as those for the moments and cumulants in the previous theorem are known as product or diagram formulas in the literature. This line of research for the Poisson case goes back to the work \cite{Surgailis}. There as well as in \cite{LastPenroseBook,LPST,PeTa,DissertationSchulte} such formulas were derived under stronger integrability assumptions. For $m$ even and $f^{(1)}=\hdots=f^{(m)}$, \eqref{eq:CumThmAss1} and \eqref{eq:CumThmAss2} only require that the integrals appearing on the right-hand side of the moment formula exist as finite numbers, whence one can regard the integrability assumptions as minimal. The related but different problem of not only computing the expectation but the whole chaos expansion of a product of multiple Wiener-It\^o integrals is studied, for example, in \cite{DoeblerPeccatiECP18,Last,PeTa,Surgailis}.

Next, we present corresponding moment and a cumulant formulas for Poisson U-statistics. 

\begin{theorem}[{Moment and cumulant formulas} for Poisson U-statistics]\label{thm:CumulantFormulaUStat}
	Let $m\in\N$, let $\widetilde{m}\in\mathbb{N}$ with $\widetilde{m}\geq m$ be even and let $f^{(\ell)}\in L^1_s(\mu^{q_\ell})$ with $q_\ell\in\N$, $\ell\in\{1,\hdots,m\}$, be such that
	\begin{align}
		\int_{\BX^{|\sigma|}}|(\otimes_{\ell=1}^{\widetilde{m}}f^{(i)})_\sigma|\,\dint\mu^{|\sigma|} &<\infty,\qquad \sigma\in\overline{\Pi}^{\widetilde{m}}(q_i),\,i\in\{1,\ldots,m\},\label{eq:UStatThmAss1}\\
		\int_{\BX^{|\sigma|}} \big|(\otimes_{\ell=1}^m f^{(\ell)})_\sigma \big| \, \dint\mu^{|\sigma|} &<\infty,\qquad \sigma \in \widetilde{\Pi}(q_1,\hdots,q_m)\,.\label{eq:UStatThmAss2}
	\end{align}
	Then, for $S_\ell:=\sum_{(x_1,\ldots,x_{q_\ell})\in\eta^{q_\ell}_{\neq}} f^{(\ell)}(x_1,\ldots,x_{q_\ell})$, $\ell\in\{1,\ldots,m\}$,
	\begin{align*}
	\BE\Big[\prod_{\ell=1}^m(S_\ell-\BE S_\ell)\Big] &= \sum_{\sigma\in\overline{\Pi}(q_1,\ldots,q_m)}\int_{\BX^{|\sigma|}}(\otimes_{\ell=1}^mf^{(\ell)})_{\sigma}\,\dint\mu^{|\sigma|}\,,\\
	\cum(S_1,\hdots,S_m) &=\sum_{\sigma\in\widetilde{\Pi}(q_1,\hdots,q_m)} \int_{\BX^{|\sigma|}} (\otimes_{\ell=1}^m f^{(\ell)})_\sigma \, \dint\mu^{|\sigma|}\,.
	\end{align*}
\end{theorem}

{Note that the difference between the formulas for the joint {moments and} cumulants of multiple Wiener-It\^o integrals in Theorem \ref{thm:MDP} and those in Theorem \ref{thm:CumulantFormulaUStat} for Poisson U-statistics is the appearance of {different sets} of partitions one has to sum over. We remark that the formulas in Theorem \ref{thm:CumulantFormulaUStat} {generalise} \cite[Proposition 12.13]{LastPenroseBook}, \cite[Corollary 1]{LPST} and \cite[Corollary 3.5]{DissertationSchulte}.}

\section{Applications}\label{sec:Applications}

\subsection{U-statistics with fixed kernel and $k$-geodesic processes}\label{subsec:ApplUstatistics}

We start our collection of applications by considering Poisson U-statistics whose kernel function does not depend on the intensity parameter of the underlying Poisson process.

\begin{corollary}\label{cor:ExFixedKernel}
Let $(\BX,\cX,\mu)$ be a probability space, let $(t_n)_{n\in\N}$ be a sequence of real numbers satisfying $t_n\geq 1$ and $t_n\to\infty$, as $n\to\infty$, and let $f: \BX^q\to \R$ for some $q\in\N$ be measurable and such that $\|f\|_\infty:=\sup\{|f(x)|: x \in\mathbb{X}^q\}<\infty$ and
\begin{align}\label{eq:FixedKernelConstantVf}
v_f:=q^2 \int_{\BX} \bigg(\int_{\BX^{q-1}} f(x,x_1,\hdots,x_{q-1}) \, \mu^{q-1}(\dint(x_1,\hdots,x_{q-1})) \bigg)^2 \, \mu(\dint x)>0\,.
\end{align}
For each $n\in\N$, let $\eta_n$ be a Poisson process on $\BX$ with intensity measure $t_n\mu$. Then the sequence $(S_n)_{n\in\N}$ given by
$$
S_n:=\sum_{(x_1,\hdots,x_q)\in \eta_{n,\neq}^q} f(x_1,\hdots,x_q)
$$
satisfies $\boldsymbol{{\rm MDP}}(q-1, (\tau_n)_{n\in\N})$, $\boldsymbol{{\rm CI}}(q-1, (\tau_n)_{n\in\N})$ and $\boldsymbol{{\rm NACC}}(q-1, (\tau_n)_{n\in\N})$ with $\tau_n:=\sqrt{t_n}/ (q^{3q}\max\{(\|f\|_\infty/\sqrt{v_f})^3,\|f\|_\infty/\sqrt{v_f}\})$. 
\end{corollary}

The previous corollary can be applied, for example, to Poisson hyperplane or Poisson $k$-flat processes in $\R^d$, one of the principal models considered in stochastic geometry, cf.\ \cite{SW08}. Following \cite{BetkenHugThaele,HeroldHugThaele,HugThaeleSTIT,KabluchkoThaeleFaces} we treat this model in greater generality and denote for $d\geq 2$ and $\kappa\in\{-1,0,1\}$ by $\MM_\kappa^d$ the $d$-dimensional standard space of constant curvature $\kappa$. As a model for $\MM_1^d$ we can take the $d$-dimensional unit sphere $\mathbb{S}^d\subset\R^{d+1}$, for $\MM_0^d$ the $d$-dimensional Euclidean space $\R^d$ and for $\MM_{-1}^d$ the Beltrami-Klein model in the interior $B^d:=\{x=(x_1,\ldots,x_d)\in\R^d:x_1^2+\ldots+x_d^2<1\}$ of the $d$-dimensional unit ball, see \cite[Chapter 6]{Ratcliffe2019}. For $k\in\{1,\ldots,d-1\}$ a $k$-geodesic of $\MM_\kappa^d$ is a totally geodesic $k$-dimensional submanifold of $\MM^d_\kappa$ and by $A_\kappa(d,k)$ we indicate the space of $k$-geodesics in $\MM_\kappa^d$. In the language of our model spaces, the elements of $A_1(d,k)$ arise as intersections of $\mathbb{S}^d$ with $(k+1)$-dimensional linear subspaces of $\R^{d+1}$, $A_0(d,k)$ is the space of $k$-dimensional affine subspaces of $\R^d$ and each $k$-geodesic in the Beltrami-Klein model for $\MM_{-1}^d$ is the non-empty intersection of $B^d$ with an element from $A_0(d,k)$. The spaces $A_\kappa(d,k)$ carry natural measures $\mu_{k,\kappa}^d$ which are invariant under the action of the full isometry group of $\MM_\kappa^d$ and are unique up to a multiplicative factor. Since $\MM_1^d$ is a compact space, we can take for $\mu_{k,1}^d$ the invariant probability measure, whereas for $\mu_{k,0}^d$ we choose the same normalization as in \cite{SW08} and for $\mu_{k,-1}^d$ the normalization as in \cite{HeroldHugThaele}.

Next, we fix a Borel set $W\subset\MM_\kappa^d$ such that $c_{k,\kappa,W}:=\mu_{k,\kappa}(\{E\in A_\kappa(d,k):E\cap W\neq\emptyset\})\in(0,\infty)$ and denote by $\mu_{k,\kappa,W}^d$ the restriction of $c_{k,\kappa,W}^{-1}\mu_{k,\kappa}^d$ to $\{E\in A_\kappa(d,k):E\cap W\neq\emptyset\}$. Now, let $(t_n)_{n\in\N}$ be a sequence such that $t_n\geq 1$ for any $n\in\N$ and $t_n\to\infty$, as $n\to\infty$, and for each $n\in\N$, let $\eta_n=\eta_{k,\kappa,W,n}$ be a Poisson process on $A_\kappa(d,k)$ with intensity measure $t_n\mu_{k,\kappa,W}^d$. In the classical Euclidean {case} $\kappa=0$, $\eta_n$ is called a Poisson $k$-flat process and we refer to \cite{SW08} for further details and background material concerning such processes and their most fundamental properties. Spherical and hyperbolic processes of $k$-geodesics were only recently studied in detail in \cite{HeroldHugThaele,HugThaeleSTIT,KabluchkoThaeleFaces} for $k=d-1$ and in \cite{BetkenHugThaele} for general $k$, respectively.

Assume now that $2k\geq d$ and let $q\in\N$ be such that $d-q(d-k)\geq 0$. The intersection process of order $q$ induced by $\eta_n$ arises by considering the intersections of any $q$ pairwise different $k$-flats from $\eta_n$. This intersection is almost surely either empty or an element of $A_\kappa(d,d-q(d-k))$, see \cite[Lemma 2.2]{BetkenHugThaele}. Now, {for} each $n\in\N$ define the Poisson U-statistic
$$
S_n := {1\over q!}\sum_{(E_1,\ldots,E_q)\in\eta_{n,\neq}^q}\mathcal{H}_\kappa^{d-q(d-k)}(E_1\cap\ldots\cap E_q\cap W)
$$
of order $q$ with fixed kernel $\mathcal{H}_\kappa^{d-q(d-k)}(\,\cdot\,\cap W)/q!$, where $\mathcal{H}_\kappa^{d-q(d-k)}$ denotes the $(d-q(d-k))$-dimensional Hausdorff measure with respect to the Riemannian metric in $\MM_\kappa^d$. In other words, $S_n$ is the total $(d-q(d-k))$-volume of the trace of the intersection process of order $q$ induced by $\eta_n$ within $W$. 

The U-statistics {$(S_n)_{n\in\N}$} satisfy the assumptions of Corollary \ref{cor:ExFixedKernel} with explicitly known constants in \eqref{eq:FixedKernelConstantVf}, see \cite[Section 6]{LPST} for the Euclidean case $\kappa=0$ and \cite{BetkenHugThaele} for general $\kappa$ (in fact, the constant $v_f$ is implicit in \cite[Proposition 3.1]{BetkenHugThaele} and $\|f\|_\infty=\max\{\mathcal{H}^{d-q(d-k)}(F\cap W):F\in A_\kappa(d,d-q(d-k))\}<\infty$). Thus, Corollary \ref{cor:ExFixedKernel} yields moderate deviation principles, concentration inequalities as well as a normal approximation bound with Cram\'er correction  in this situation.

\subsection{Subgraph counts in random geometric graphs}\label{subsec:ApplRandomGeometricGraphs}

Let $(t_n)_{n\in\N}$ be a positive real-valued sequence such that $t_n\to\infty$, as $n\to\infty$, and let $(r_n)_{n\in\N}$ be a positive real-valued sequence. Further, we fix a compact convex set $W\subset\R^d$ with interior points and let, for each $n\in\N$, $\eta_n$ be a Poisson process whose intensity measure $\mu_n$ is $t_n$ times the {restriction of the Lebesgue measure to} $W$. Based on this data, we construct the random geometric graph ${\rm RGG}(\eta_n,r_n)$ by taking the points of $\eta_n$ as the vertices of the graph and by connecting two distinct points by an edge whenever their (Euclidean) distance is strictly positive and does not exceed the given threshold $r_n$.

We are interested in the subgraph counting statistics associated with the random geometric graph. To define them, let $G$ be a fixed connected graph with $q$ vertices, and for $x_1,\ldots,x_q\in W$ we let $f_{=}(x_1,\ldots,x_q; G,r_n)$ be $1/q!$ times the indicator that the random geometric graph ${\rm RGG}(\{x_1,\ldots,x_q\},r_n)$ is isomorphic to $G$, while $f_{\subset}(x_1,\ldots,x_q; G,r_n)$ is $1/q!$ times the number of subgraphs of ${\rm RGG}(\{x_1,\ldots,x_q\},r_n)$ that are isomorphic to $G$. The subgraph counting statistics are given by
$$
S^{\diamondsuit}_n(G) := \sum_{(x_1,\ldots,x_q)\in\eta_{n,\neq}^q}f_{\diamondsuit}(x_1,\ldots,x_q;G,r_n)\,, \quad \diamondsuit\in\{=,\subset\}.
$$
Here, $S^{=}_n(G)$ and $S^{\subset}_n(G)$ are the numbers of induced and non-induced copies of $G$ in ${\rm RGG}(\eta_n,r_n)$. For example, if $G$ is the graph with three vertices and two edges, three vertices connected by three edges in ${\rm RGG}(\eta_n,r_n)$ are counted thrice in $S^{\subset}_n(G)$ but do not to contribute to $S^{=}_n(G)$. From Theorem \ref{thm:MDP2} we can deduce the following result.

\begin{corollary}\label{cor:ExRGG}
For $k\in\N$, $a_1,\hdots,a_k\in\R\setminus\{0\}$, $\diamondsuit_1,\hdots,\diamondsuit_k\in\{=,\subset\}$ and connected graphs $G_1,\hdots,G_k$ such that $G_i$ has $q_i$ vertices for $i\in\{1,\ldots,k\}$, let $S_n:=\sum_{i=1}^k a_i S_n^{\diamondsuit_i}(G_i)$, and define $p:=\min\{q_1,\hdots,q_k\}$, $q:=\max\{q_1,\hdots,q_k\}$ and $a:=\max\{|a_1|,\hdots,|a_k|\}$.
Assume that $t_n\to\infty$, as $n\to\infty$, and that there is a constant $v>0$ such that
\begin{equation}\label{eqn:vRGG}
\BV S_n \geq  v \max\{t_n^{2q-1}(\kappa_dr_n^d)^{2q-2},t_n^{p}(\kappa_dr_n^d)^{p-1}\}\,, \qquad n\in\N\,.
\end{equation}
Then the sequence $(S_n)_{n\in\N}$ satisfies $\boldsymbol{{\rm MDP}}(q-1,(\tau_n)_{n\in\N})$, $\boldsymbol{{\rm CI}}(q-1,(\tau_n)_{n\in\N})$ and $\boldsymbol{{\rm NACC}}(q-1,(\tau_n)_{n\in\N})$ with $\tau_n$ given by
$$
\tau_n:= \frac{\sqrt{v t_n \min\{1,t_n\kappa_dr_n^d\}^{p-1}}}{(kq^q)^3\max\{a,a^3\Vol(W)/v\}}\,,\qquad n\in\N\,.
$$
\end{corollary}

Our concentration inequality for the subgraph counting statistic is not the first one in this direction that appeared in the literature. More precisely, for a fixed connected graph $G$ with $q\geq 2$ vertices, {it was shown in \cite[Theorem 1.1]{BachmannReitzner}} that, essentially, the subgraph counting statistic $S_n^\subset(G)$ satisfies a concentration inequality of the form
$$
\BP(|S_n^\subset(G)-\BE S^\subset_n(G)|\geq z)\leq 2\exp(-c_{d,q}z^{1/q})\,,\qquad z\geq 0\,,
$$
where $c_{d,q}\in(0,\infty)$ is a constant that only depends on the space dimension $d$ as well as on the vertex number $q$ of $G$. We emphasise that our Corollary \ref{cor:ExRGG} yields a result of a similar nature and especially shows the same exponent $1/q$ for $z$ (for the special case of the edge counting statistic corresponding to the choice $q=2$ we refer to \cite{BachmannPeccati,ReitznerSchulteThaele} and the discussions therein). On the other hand, our framework allows to deal simultaneously with a finite number of graphs and also leads to moderate deviation principles and normal approximation bounds with Cram\'er correction, which have no counterparts in the existing literature.

\begin{remark}\rm 
The assumption on the existence of a strictly positive constant $v>0$ in the lower variance bound \eqref{eqn:vRGG} for $S_n$ is always satisfied if we choose $k=1$ and $\diamondsuit_1=\text{`$\subset$'}$. For $\diamondsuit_1=\text{`$=$'}$ we additionally need to assume that $G_1$ is feasible in the sense that $\BP(\operatorname{RGG}(\{X_1,\hdots,X_{q_1}\},r)\text{ is isomorphic to }G_1)>0$ for some $r>0$ and independent random points $X_1,\hdots,X_{q_1}$ uniformly distributed on $W$, cf.\ \cite[Chapter 3]{PenroseBook} for a discussion of the latter concept. This follows, for example, from the results in \cite[Chapter 3.3]{PenroseBook} and especially from the remark after Proposition 3.7 there. On the other hand, if $k\geq 2$ the assumption seems unavoidable and does not need to be satisfied in general. For example, if $k=2$ and if we take $G_1=G_2=G$ for some connected graph $G$ (on $q\geq 2$ vertices) as well as $a_1=-a_2$, then $S_n$ is identically zero with probability one and so \eqref{eqn:vRGG} is satisfied only with $v=0$.
\end{remark}

\begin{remark}\rm 
A generalisation of random geometric graphs is the random connection model, where similarly to the Erd\H{o}s-R\'enyi random graph it is decided independently for each pair of points of the underlying Poisson process if they are connected by an edge, but the probability of an edge depends on the relative spatial position of the vertices. In the recent preprint \cite{LiuPrivault}, for subgraph counts of random connection models cumulant bounds are derived in order to establish rates of convergence for the normal approximation in Kolmogorov distance. Although the random geometric graph is a special case of the random connection model, the cumulant estimates required for our Corollary \ref{cor:ExRGG} do not follow directly from the findings of \cite{LiuPrivault}, since this paper works with rescalings of a fixed connection function, which does not cover the situation of Corollary \ref{cor:ExRGG} with the two parameters $t$ and $r_t$. However, the moment and cumulant formulas for subgraph counts in \cite{LiuPrivault} are very similar to our general formulas for Poisson U-statistics, since subgraph counts can be seen as U-statistics of a Poisson process with an additional randomisation coming from drawing edges randomly. In particular, the same classes of partitions are used. The lower bound in our Proposition \ref{Prop:tildePi} implies that the exponent $r$ in Lemma 2.6 of \cite{LiuPrivault} is in fact optimal.
\end{remark}

\subsection{Quadratic functionals of Ornstein-Uhlenbeck L\'evy processes}\label{subsec:ApplOrnsteinUhlenbeck}

Let $\eta$ be a Poisson process on $\R\times\R$ with intensity measure $\mu:=\lambda\otimes \nu$, where $\lambda$ is the Lebesgue measure on $\R$ and $\nu$ is a $\sigma$-finite measure on $\R$ with $\int_\R u^2 \, \nu(\dint u) =1$. We denote by $\hat{\eta}$ the compensated Poisson process $\eta-\mu$. For a fixed parameter $\varrho>0$ the Ornstein-Uhlenbeck L\'evy process $(U_t)_{t\in\R}$ is given as the stochastic integral
$$
U_t:=\sqrt{2\varrho} \int_{(-\infty,t]\times \R} u e^{-\varrho (t-x)} \, \hat{\eta}(\dint (x,u))
$$
with respect to $\hat{\eta}$. We are interested in the behaviour of the quadratic functionals
$$
Q(T):= \int_0^T U_t^2 \, \dint t, \quad T\geq 0\,.
$$
A central limit theorem for $Q(T)$ with a rate of convergence has been obtained in \cite{PeccatiSoleTaqquUtzet} in terms of the Wasserstein distance and in terms of the Kolmogorov distance in \cite{EichelsbacherThaele}. Our next result adds a moderate deviation principle, a concentration inequality as well as a normal approximation bound with Cram\'er correction. To formulate it, define the constant $c_\nu:=\int_{-\infty}^\infty u^4 \, \nu(\dint u)$.

\begin{corollary}\label{cor:ExOrnsteinUhlenbeck}
Let $(T_n)_{n\in\N}$ be a positive real-valued sequence such that $T_n>1/\varrho$ and $T_n\to\infty$, as $n\to\infty$, and assume that there exists a constant $M\geq 1$ such that $\int_{-\infty}^\infty |u|^m \, \nu(\dint u)\leq M^m$ for all $m\in\N$. Then the sequence $(Q(T_n))_{n\in\N}$ satisfies $\boldsymbol{{\rm MDP}}(1,(\tau_n)_{n\in\N})$, $\boldsymbol{{\rm CI}}(1,(\tau_n)_{n\in\N})$ and $\boldsymbol{{\rm NACC}}(1,(\tau_n)_{n\in\N})$ with $(\tau_n)_{n\in\N}$ given by
$$
\tau_n:= \frac{1}{512}\min\bigg\{1,  \frac{c_\nu\big(T_n-1/\varrho \big)}{(4M^4)^2\max\{1,2/\varrho\}^2 (T_n+1/(2\varrho))}\bigg\} \frac{\sqrt{c_\nu\big(T_n-1/\varrho \big)}}{4M^4\max\{1,2/\varrho\}}\,.
$$
\end{corollary}

\section{Proofs: Product formulas}\label{sec:ProofProducFormula}

The proof of Theorem \ref{thm:CumulantFormulaMultipleIntegrals} is prepared by the following lemma. Since the measure $\mu$ was assumed to be $\sigma$-finite, there exists a sequence $(A_n)_{n\in\N}$ of measurable subsets of $\BX$ with $A_n\subseteq A_{n+1}$ and $\mu(A_n)<\infty$ for $n\in\N$ satisfying $\bigcup_{n\in\N}A_n=\BX$. Now, for a function {$f:\BX^q\to\R$ with $q\in\N$ and $n\in\N$ define its truncation $f_n: \BX^q\to\R$ by
\begin{equation}\label{eq:DefFn}
f_n(x) := {\bf 1}_{A_n^q}(x){\bf 1}_{[-n,n]}(f(x))f(x),\qquad x\in\BX^q.
\end{equation}}

\begin{lemma}\label{lem:ProdForm}
Fix an even number $m\in\mathbb{N}$  and {$f\in L^2_s(\mu^q)$} for some $q\in\N$. Assume that
\begin{equation}\label{eqn:assumption_ProdForm}
\int_{\BX^{|\sigma|}}|(\otimes_{\ell=1}^mf)_\sigma|\,{\rm d}\mu^{|\sigma|}<\infty
\end{equation}
for all $\sigma\in\Pi_{\geq 2}^m(q)$. Then $\lim\limits_{n\to\infty}\BE I_q(f-f_n)^m = 0$ and
$$
\lim_{n\to\infty}\BE I_q(f_n)^m = \BE I_q(f)^m = \sum_{\sigma\in\Pi_{\geq 2}^m(q)}\int_{\BX^{|\sigma|}}(\otimes_{\ell=1}^m f)_\sigma\,\dint\mu^{|\sigma|}<\infty.
$$
\end{lemma}
\begin{proof}

We note that the sequence of functions {$(f_n)_{n\in\mathbb{N}}$} satisfies the following properties:
\begin{itemize}
\item each $f_n$ is bounded and has support of finite $\mu$-measure,
\item $f_n$ converges to $f$ as $n\to\infty$ $\mu$-almost everywhere,
\item $|f_n|\leq|f|$.
\end{itemize}
By the dominated convergence theorem, it follows that $\|f-f_n\|^2_{L^2(\mu^q)}\to 0$, as $n\to\infty$. {Now,} let $\varepsilon>0$ and use Chebychev's inequality to see that
\begin{align*}
\BP(|I_q(f)-I_q(f_n)|\geq\varepsilon) = \BP(|I_q(f-f_n)|\geq\varepsilon) \leq {\BE I_q(f-f_n)^2\over\varepsilon^2} = \frac{q!\|f-f_n\|_{{L^2(\mu^q)}}^2}{\varepsilon^2},
\end{align*}
{where} the last expression tends to $0$, as $n\to\infty$. Hence, $I_q(f_n)$ converges to $I_q(f)$ in probability and in distribution, {as $n\to\infty$.} Thus, 
\begin{align*}
\BE I_q(f)^m \leq \liminf_{n\to\infty}\BE I_q(f_n)^m
\end{align*}
by the Portmanteau theorem. Note that the expression $\BE I_q(f)^m$ on the left-hand side is in fact well defined since we have assumed $m$ to be even. To $\BE I_q(f_n)^m$ we can now apply \cite[Theorem 1]{LPST} which yields
\begin{align*}
\BE I_q(f)^m \leq \liminf_{n\to\infty}\sum_{\sigma\in\Pi_{\geq 2}^m(q)}\int_{\BX^{|\sigma|}}(\otimes_{\ell=1}^mf_n)_\sigma\,\dint\mu^{|\sigma|} = \sum_{\sigma\in\Pi_{\geq 2}^m(q)}\int_{\BX^{|\sigma|}}(\otimes_{\ell=1}^mf)_\sigma\,\dint\mu^{|\sigma|},
\end{align*}
where the last step is verified by the dominated convergence theorem, which is applicable due to \eqref{eqn:assumption_ProdForm}. Consequently, replacing $f$ by $f-f_n$, we obtain
\begin{align*}
\BE I_q(f-f_n)^m \leq \sum_{\sigma\in\Pi_{\geq 2}^m(q)}\int_{\BX^{|\sigma|}}(\otimes_{\ell=1}^m(f-f_n))_\sigma\,\dint\mu^{|\sigma|}
\end{align*}
for all $n\in\N$. Once again by the dominated convergence theorem and \eqref{eqn:assumption_ProdForm} the right-hand side tends to $0$, as $n\to\infty$. This proves the first assertion.

On the other hand, using the reverse triangle inequality we see that
$$
|(\BE I_q(f)^m)^{1/m}-(\BE I_q(f_n)^m)^{1/m}| \leq (\BE I_q(f-f_n)^m)^{1/m}.
$$
Thus, by the convergence of the right-hand side to zero and by \cite[Theorem 1]{LPST},
\begin{align*}
\BE I_q(f)^m &= \lim_{n\to\infty}\BE I_q(f_n)^m \\
&= \lim_{n\to\infty}\sum_{\sigma\in\Pi_{\geq 2}^m(q)}\int_{\BX^{|\sigma|}}(\otimes_{\ell=1}^mf_n)_\sigma\,\dint\mu^{|\sigma|}\\
& = \sum_{\sigma\in\Pi_{\geq 2}^m(q)}\int_{\BX^{|\sigma|}}(\otimes_{\ell=1}^mf)_\sigma\,\dint\mu^{|\sigma|},
\end{align*}
where the last step follows once more by the dominated convergence theorem. This completes the argument.
\end{proof}

\begin{proof}[Proof of Theorem \ref{thm:CumulantFormulaMultipleIntegrals}]
For each $\ell\in\{1,\ldots,m\}$ and $n\in\N$ define the function $f_n^{(\ell)}$ as in \eqref{eq:DefFn}. Then, using \cite[Theorem 1]{LPST} in the first and the dominated convergence theorem, {which is applicable due to \eqref{eq:CumThmAss2},} in the second step, we have that
\begin{align*}
\lim_{n\to\infty}\BE\Big[\prod_{\ell=1}^mI_{q_\ell}(f_n^{(\ell)})\Big] &= \lim_{n\to\infty}\sum_{\sigma\in{\Pi}_{\geq 2}(q_1,\hdots,q_m)} \int_{\BX^{|\sigma|}} (\otimes_{\ell=1}^m f_n^{(\ell)})_\sigma \, \dint\mu^{|\sigma|}\\
&=\sum_{\sigma\in{\Pi}_{\geq 2}(q_1,\hdots,q_m)} \int_{\BX^{|\sigma|}} (\otimes_{\ell=1}^m {f^{(\ell)}} )_\sigma \, \dint\mu^{|\sigma|}.
\end{align*}
From Lemma \ref{lem:ProdForm}, {whose assumption is satisfied by \eqref{eq:CumThmAss1},} one {deduces} that, for $\ell\in\{1,\ldots,m\}$,
\begin{equation}\label{eqn:convergence_Ito-integrals}
\lim_{n\to\infty}\BE I_{q_\ell}(f^{(\ell)}-f_n^{(\ell)})^{\widetilde{m}} = 0\qquad\text{and}\qquad\lim_{n\to\infty}\BE I_{q_\ell}(f_n^{(\ell)})^{\widetilde{m}} = \BE I_{q_\ell}(f^{(\ell)})^{\widetilde{m}}<\infty.
\end{equation}
Thus, using H\"older's inequality, we find that
\begin{align*}
&\lim_{n\to\infty} \Big|\BE\Big[\prod_{\ell=1}^mI_{q_\ell}(f^{(\ell)})\Big] - \BE\Big[\prod_{\ell=1}^mI_{q_\ell}(f_n^{(\ell)})\Big] \Big|\\
& \leq \lim_{n\to\infty}\sum_{\ell=1}^m\BE\Big|\prod_{j=1}^{\ell-1}I_{q_j}(f^{(j)})\times I_{q_\ell}(f^{(\ell)}-f_n^{(\ell)})\times\prod_{j=\ell+1}^mI_{q_j}(f_n^{(j)})\Big|\\
&\leq\lim_{n\to\infty} \sum_{\ell=1}^m \prod_{j=1}^{\ell-1}(\BE|I_{q_j}(f^{(j)})^m|)^{1/m}\times(\BE|I_{q_\ell}(f^{(\ell)}-f_n^{(\ell)})^m|)^{1/m}\times\prod_{j=\ell+1}^{m}(\BE|I_{q_j}(f_n^{(j)})^m|)^{1/m}\\
&\leq\lim_{n\to\infty} \sum_{\ell=1}^m  \prod_{j=1}^{\ell-1}(\BE I_{q_j}(f^{(j)})^{\widetilde{m}})^{1/\widetilde{m}}\times(\BE I_{q_\ell}(f^{(\ell)}-f_n^{(\ell)})^{\widetilde{m}})^{1/\widetilde{m}}\times\prod_{j=\ell+1}^{m}(\BE I_{q_j}(f_n^{(j)})^{\widetilde{m}})^{1/\widetilde{m}},
\end{align*}
where in the last step we used that $\widetilde{m}\geq m$ is even. By \eqref{eqn:convergence_Ito-integrals}, the two products ranging over $j\in\{1,\ldots,\ell-1\}$ and $j\in\{\ell+1,\ldots,m\}$, respectively, converge to finite constants, while the middle term $(\BE I_{q_\ell}(f^{(\ell)}-f_n^{(\ell)})^{\widetilde{m}})^{1/\widetilde{m}}$ tends to zero, as $n\to\infty$. This completes the proof of the first part of Theorem \ref{thm:CumulantFormulaMultipleIntegrals}.

The formula for the joint cumulants eventually follows from this exactly as in the proof of Theorem 1 in \cite{LPST}.
\end{proof}

As a preparation for the proof of Theorem \ref{thm:CumulantFormulaUStat} we need to introduce an operation on partitions. Fix integers $m\geq 1$, $q_1,\ldots,q_m,j_1,\ldots,j_m\geq 1$ with $j_1\leq q_1,\ldots,j_m\leq q_m$. First, define $\tau_{j_1,\ldots,j_m}^{q_1,\ldots,q_m}:\{1,\ldots,j_1+\ldots+j_m\}\to\{1,\ldots,q_1+\ldots+q_m\}$ by putting 
$$
\tau_{j_1,\ldots,j_m}^{q_1,\ldots,q_m}(u):=\sum_{i=1}^{k-1}q_i+r_k\qquad\text{if } u=\sum_{i=1}^{k-1}j_i+r_k,\;r_k\leq j_k\text{ for some }k\in\{1,\ldots,m\}.
$$
Then we define the mapping $\Theta_{j_1,\ldots,j_m}^{q_1,\ldots,q_m}:{\Pi}(j_1,\ldots,j_m)\to\Pi(q_1,\ldots,q_m)$, which assigns to a partition $\pi=\{B_1,\ldots,B_v\}\in{\Pi}(j_1,\ldots,j_m)$ with $v\geq 1$ blocks the partition
$$
\Theta_{j_1,\ldots,j_m}^{q_1,\ldots,q_m}(\pi):=\{\tau(B_1),\ldots,\tau(B_v)\}\cup\bigcup_{z\in\{1,\ldots,q_1+\ldots+q_m\}\atop z\notin\cup_{w=1}^v\tau(B_w)}\{\{z\}\},\qquad \tau:=\tau_{j_1,\ldots,j_m}^{q_1,\ldots,q_m}.
$$
Let us explain the mechanism behind the mapping $\Theta_{j_1,\ldots,j_m}^{q_1,\ldots,q_m}$.  The partition $\sigma:=\Theta_{j_1,\ldots,j_m}^{q_1,\ldots,q_m}(\pi)\in\Pi(q_1,\ldots,q_m)$ arises from $\pi\in\Pi(j_1,\ldots,j_m)$ by first adding to the first row corresponding to the graphical representation of $\pi$ the elements $j_1+1,\ldots,q_1$, each of which becomes an individual block of $\sigma$. In the second row corresponding to $\pi$ we first shift the labels of the elements there by $q_1-j_1$ and then add the elements $q_1+j_1+1,\ldots,q_1+q_2$, which again become new individual blocks of $\sigma$. In the third row the elements are shifted by $q_1-j_1+q_2-j_2$ and then new elements $q_1+q_2+j_3+1,\ldots,q_1+q_2+q_3$ are added as individual blocks etc. The procedure is illustrated in Figure \ref{fig:2} by an example.

\begin{figure}[t]
\begin{center}
\includegraphics[width=0.6\columnwidth]{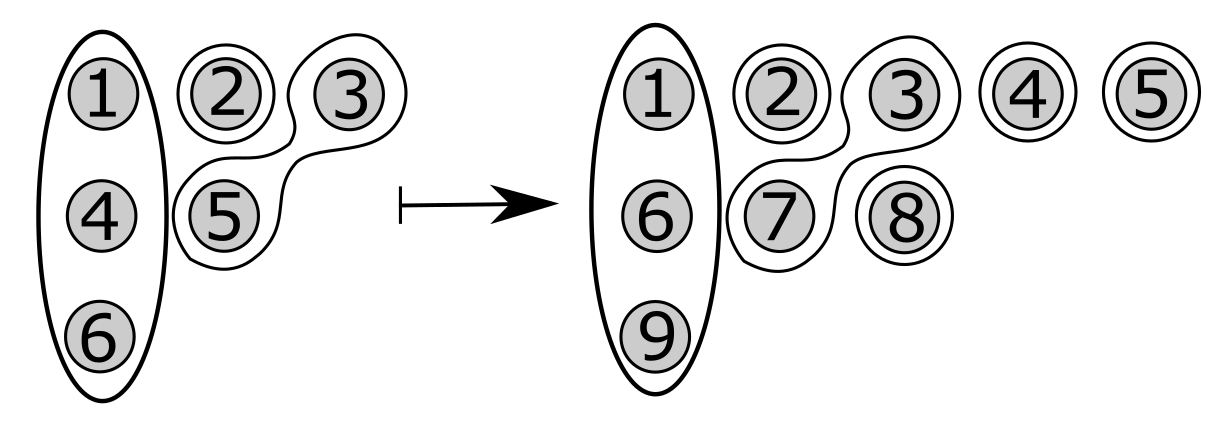}
\caption{\small Illustration of the mapping $\Theta_{3,2,1}^{5,3,1}:\Pi(3,2,1)\to\Pi(5,3,1)$.}
\label{fig:2}
\end{center}
\end{figure}

\begin{proof}[Proof of Theorem \ref{thm:CumulantFormulaUStat}]
Fix $i\in\{1,\ldots,m\}$. Applying \eqref{eqn:ChaosExpansionUstatistic} to the Poisson U-statistic $S_i$ we have that
\begin{equation}\label{eq:ChaosUStat}
S_i-\BE S_i = \sum_{j=1}^{q_i}I_j(f_j^{(i)})
\end{equation}
with functions $f_j^{(i)}:\BX^j\to\R$ given by
$$
f_j^{(i)}(x_1,\ldots,x_j) := {q_i\choose j}\int_{\BX^{q_i-j}}f^{(i)}(x_1,\ldots,x_j,x_{j+1},\ldots,x_{q_i})\,\mu^{q_i-j}(\dint(x_{j+1},\ldots,x_{q_i}))
$$
{for $j\in\{1,\hdots,q_i\}$.} By \eqref{eq:UStatThmAss1} and since $\widetilde{m}$ is even, we have that $f_j^{(i)}\in L^2_s(\mu^j)$ for each $j\in\{1,\ldots,q_i\}$ and that $f_j^{(i)}$ satisfies \eqref{eq:CumThmAss1}. Indeed, for every {$\pi\in\Pi_{\geq 2}^{\widetilde{m}}(j)$} one has that
$$
\int_{\BX^{|\pi|}}|(\otimes_{\ell=1}^{\widetilde{m}}f_j^{(i)})_\pi|\,\dint\mu^{|\pi|} \leq {q_i\choose j}^{\widetilde{m}}\int_{\BX^{|\sigma|}}|(\otimes_{\ell=1}^{\widetilde{m}}f^{(i)})_\sigma|\,\dint\mu^{|\sigma|},
$$
where we applied the definition of $f_j^{(i)}$, used the triangle inequality and put $\sigma:=\Theta_{j,\ldots,j}^{q_i,\ldots,q_i}(\pi)\in\overline{\Pi}^{\widetilde{m}}(q_i)$. Similarly, for $j_\ell\in\{1,\ldots,q_\ell\}$ with $\ell\in\{1,\ldots,m\}$ and $\pi\in\Pi_{\geq 2}(j_1,\ldots,j_m)$ we have that
\begin{align*}
\int_{\BX^{|\pi|}}|(\otimes_{\ell=1}^m f_{j_\ell}^{(\ell)})_\pi|\,\dint\mu^{|\pi|}\leq \prod_{\ell=1}^m{q_\ell\choose j_\ell}\int_{\BX^{|\sigma|}}|(\otimes_{\ell=1}^mf^{(\ell)})_\sigma|\,\dint\mu^{|\sigma|}
\end{align*}
for $\sigma:=\Theta_{j_1,\ldots,j_m}^{q_1,\ldots,q_m}(\pi)\in\overline{\Pi}(q_1,\hdots,q_m)$. Thus, the functions $f_{j_\ell}^{(\ell)}$, $\ell\in\{1,\ldots,m\}$, also satisfy \eqref{eq:CumThmAss2} due to \eqref{eq:UStatThmAss2}.

We can now repeatedly apply Theorem \ref{thm:CumulantFormulaMultipleIntegrals} to \eqref{eq:ChaosUStat} to conclude that
\begin{align*}
&\BE\Big[\prod_{\ell=1}^m(S_\ell-\BE S_\ell)\Big]\\
&= \sum_{1\leq j_1\leq q_1}\cdots\sum_{1\leq j_m\leq q_m}\BE\Big[\prod_{\ell=1}^m {I_{j_\ell}(f_{j_\ell}^{(\ell)})} \Big]\\
&=\sum_{1\leq j_1\leq q_1}\cdots\sum_{1\leq j_m\leq q_m}\sum_{\pi\in\Pi_{\geq 2}({j_1},\ldots,{j_m})}\int_{\BX^{|\pi|}}(\otimes_{\ell=1}^mf_{j_\ell}^{(\ell)})_\pi\,\dint\mu^{|\pi|}\\
&=\sum_{1\leq j_1\leq q_1}\cdots\sum_{1\leq j_m\leq q_m}\Big(\prod_{\ell=1}^m{q_\ell\choose j_\ell}\Big)\\
&\qquad\qquad\times\sum_{\pi\in\Pi_{\geq 2}({j_1},\ldots,{j_m})}\int_{\BX^{|\Theta_{j_1,\ldots,j_m}^{q_1,\ldots,q_m}(\pi)|}}(\otimes_{\ell=1}^mf^{(\ell)})_{\Theta_{j_1,\ldots,j_m}^{q_1,\ldots,q_m}(\pi)}\,\dint\mu^{|\Theta_{j_1,\ldots,j_m}^{q_1,\ldots,q_m}(\pi)|}.
\end{align*}
Next, we note that $\Theta_{j_1,\ldots,j_m}^{q_1,\ldots,q_m}(\pi)$ has $q_\ell-j_\ell$ singleton blocks at the end of row $\ell\in\{1,\ldots,m\}$ in the representing diagram. Moreover, the prefactor $\prod_{\ell=1}^m{q_\ell\choose j_\ell}$ is precisely the number {of} possibilities to choose $q_\ell-j_\ell$ singletons in each row $\ell\in\{1,\ldots,m\}$ of a diagram corresponding to a partition in $\overline{\Pi}(q_1,\ldots,q_m)$. This together with the symmetry of the functions $f^{(1)},\ldots,f^{(m)}$ implies that
\begin{align*}
&\Big(\prod_{\ell=1}^m{q_\ell\choose j_\ell}\Big)\sum_{\pi\in\Pi_{\geq 2}({j_1},\ldots,{j_m})}\int_{\BX^{|\Theta_{j_1,\ldots,j_m}^{q_1,\ldots,q_m}(\pi)|}}(\otimes_{\ell=1}^mf^{(\ell)})_{\Theta_{j_1,\ldots,j_m}^{q_1,\ldots,q_m}(\pi)}\,\dint\mu^{|\Theta_{j_1,\ldots,j_m}^{q_1,\ldots,q_m}(\pi)|} \\
&\qquad =\sum_{\substack{\sigma\in\overline{\Pi}({q_1},\ldots,{q_m})\\ \sigma\text{ has $q_1-j_1$ singletons in row $1$}\\\cdots\\ \sigma\text{ has $q_m-j_m$ singletons in row $m$}}}\int_{\BX^{|\sigma|}}(\otimes_{\ell=1}^mf^{(\ell)})_{\sigma}\,\dint\mu^{|\sigma|}.
\end{align*}
Thus,
\begin{align*}
\BE\Big[\prod_{\ell=1}^m(S_\ell-\BE S_\ell)\Big] = \sum_{\sigma\in\overline{\Pi}(q_1,\ldots,q_m)}\int_{\BX^{|\sigma|}}(\otimes_{\ell=1}^mf^{(\ell)})_{\sigma}\,\dint\mu^{|\sigma|},
\end{align*}
proving the first part of the theorem.

For the second part we use \eqref{eq:ChaosUStat}, multilinearity of joint cumulants and apply Theorem \ref{thm:CumulantFormulaMultipleIntegrals} to see that
\begin{align*}
&\cum(S_1,\hdots,S_m)\\
&= \sum_{1\leq j_1\leq q_1}\cdots\sum_{1\leq j_m\leq q_m}\cum(I_{j_1}(f_{j_1}^{(1)}),\hdots,I_{j_m}(f_{j_m}^{(m)})) \\
&=\sum_{1\leq j_1\leq q_1}\cdots\sum_{1\leq j_m\leq q_m}\sum_{\pi\in\widetilde{\Pi}_{\geq 2}(j_1,\hdots,j_m)} \int_{\BX^{|\pi|}} (\otimes_{\ell=1}^m f_{j_\ell}^{(\ell)})_\pi \, \dint\mu^{|\pi|} \allowdisplaybreaks\\
&=\sum_{1\leq j_1\leq q_1}\cdots\sum_{1\leq j_m\leq q_m}\Big(\prod_{\ell=1}^m{q_\ell\choose j_\ell}\Big)\\
&\qquad\qquad\times\sum_{\pi\in\widetilde{\Pi}_{\geq 2}(j_1,\hdots,j_m)} \int_{\BX^{|\Theta_{j_1,\ldots,j_m}^{q_1,\ldots,q_m}(\pi)|}} (\otimes_{\ell=1}^m f^{(\ell)})_{\Theta_{j_1,\ldots,j_m}^{q_1,\ldots,q_m}(\pi)} \, \dint\mu^{|\Theta_{j_1,\ldots,j_m}^{q_1,\ldots,q_m}(\pi)|} \allowdisplaybreaks \\
&=\sum_{1\leq j_1\leq q_1}\cdots\sum_{1\leq j_m\leq q_m}\sum_{\substack{\sigma\in\widetilde{\Pi}({q_1},\ldots,{q_m})\\ \sigma\text{ has $q_1-j_1$ singletons in row $1$}\\\cdots\\ \sigma\text{ has $q_m-j_m$ singletons in row $m$}}}\int_{\BX^{|\sigma|}}(\otimes_{\ell=1}^mf^{(\ell)})_{\sigma}\,\dint\mu^{|\sigma|}\\
&=\sum_{\sigma\in\widetilde{\Pi}(q_1,\hdots,q_m)} \int_{\BX^{|\sigma|}} (\otimes_{\ell=1}^m f^{(\ell)})_\sigma \, \dint\mu^{|\sigma|}\,,
\end{align*}
where we used the same arguments as above.
This completes the proof.
\end{proof}

\section{Proofs: Moderate deviation estimates}\label{sec:ProofsMainResults}

The first aim of this section is to upper bound the right-hand sides of the cumulant formulas in Theorems \ref{thm:CumulantFormulaMultipleIntegrals} and \ref{thm:CumulantFormulaUStat} in such a way that Proposition \ref{prop:CumulantBoundImplyMDPandConcentration} can be applied. For that purpose, the following estimates are going to be important. Recall the definitions of the classes $\Pi^m(q)$ and $\widetilde{\Pi}^m_{\geq 2}(q)$ of partitions from Subsection \ref{subsec:Partitions}.

\begin{proposition}\label{Prop:tildePi}
For any $q\geq 2$ one has
\begin{equation}\label{eq:UpperBoundPartitions}
|\Pi^m(q)| \leq {q^{qm}} (m!)^q\,, \qquad m\in\N\,.
\end{equation}
Moreover, there are no constants $c>0$ and $\gamma\in (0,q)$ such that
\begin{equation}\label{eq:LowerBoundPartitions}
|\widetilde{\Pi}^m_{\geq 2}(q)| \leq c^m (m!)^{\gamma}\,,\qquad m\in\N\,.
\end{equation}
\end{proposition}

\begin{proof}
We can construct each partition {$\sigma\in \Pi^{m+1}(q)$} by taking a partition {$\hat{\sigma}\in\Pi^{m}(q)$} and deciding for each $i\in\{qm+1,\hdots,q(m+1)\}$ whether it joins an existing block or forms a new block on its own. Since $\hat{\sigma}$ has at most $qm$ blocks, we obtain from a partition $\hat{\sigma}\in\Pi_m(q)$ at most $(qm+1)^q$ new partitions. This implies the inequality
$$
|\Pi^{m+1}(q)| \leq (qm+1)^q |\Pi^m(q)| \leq q^q (m+1)^q |\Pi^m(q)|\,.
$$
Putting $C:=\max\{|\Pi^1(q)|,q^q\}$, this yields {$|\Pi^m(q)| \leq C^m (m!)^q$ for $m\in\N$ by iteration. Because of $|\Pi^1(q)|=1$, this is \eqref{eq:UpperBoundPartitions}.}

In order to show the second part, we prove that for any $\tilde{\gamma}\in(0,q)$ there exists a constant $\hat{c}>0$ such that
\begin{equation}\label{eq:ProofLowerBound}
\limsup_{m\to\infty}\frac{|\widetilde{\Pi}^m_{\geq 2}(q)|}{\hat{c}^m (m!)^{\tilde{\gamma}}}>1\,.
\end{equation}
This in turn implies \eqref{eq:LowerBoundPartitions} by the following consideration. Assume that $\gamma\in(0,q)$ and $c>0$ satisfy \eqref{eq:LowerBoundPartitions}. For $\tilde{\gamma}:=(\gamma+q)/2$ there is a constant $\hat{c}>0$ such that
$$
1<\limsup_{m\to\infty}\frac{|\widetilde{\Pi}^m_{\geq 2}(q)|}{\hat{c}^m (m!)^{\tilde{\gamma}}}=\limsup_{m\to\infty}\frac{|\widetilde{\Pi}^m_{\geq 2}(q)|}{c^m (m!)^\gamma (\hat{c}/c)^m (m!)^{\tilde{\gamma}-\gamma}}\,.
$$
Since, by construction, $(\hat{c}/c)^m (m!)^{\tilde{\gamma}-\gamma}\to\infty$, as $m\to\infty$, this means that
$$
\limsup_{m\to\infty} \frac{|\widetilde{\Pi}_{\geq 2}^m(q)|}{c^m (m!)^\gamma}>1\,,
$$
which contradicts \eqref{eq:LowerBoundPartitions}.

So, let us prove \eqref{eq:ProofLowerBound} for a given $\tilde{\gamma}\in(0,q)$. Let $k,u\in\N$ with $k$ even and put $m=uk$. In this situation we can construct partitions $\sigma\in\widetilde{\Pi}_{\geq 2}^{uk}(q)$ in the following way:
\begin{itemize}
\item [1)] We group $\{1+(\ell-1) q: \ell\in\{1,\hdots,uk\}\}$ into $u$ blocks of size $k$.
\item [2)] Now, we order these blocks as follows: $\{qi_1^{(1)}+1,\hdots,qi_k^{(1)}+1\},\hdots,\{q i_1^{(u)}+1,\hdots,qi_k^{(u)}+1\}$ with $i_1^{(1)}<i_1^{(2)}<\hdots<i_1^{(u)}$ and $i_1^{(\ell)}<i_2^{(\ell)}<\hdots<i_k^{(\ell)}$ for $\ell\in\{1,\hdots,u\}$. Next, we add the blocks $\{qi_1^{(\ell)}+2,qi_k^{(\ell+1)}+2\}$, $\ell\in\{1,\hdots,u-1\}$, and $\{qi_1^{(u)}+2,qi_k^{(1)}+2\}$. Then we form blocks of size $k-2$ from the remaining elements of $\{2+(\ell-1)q: \ell\in\{1,\hdots,uk\}\}$ in an arbitrary way.
\item [3)] For any $j\in\{3,\hdots, q\}$ we group $\{j+(\ell-1)q:\ell\in\{1,\hdots,uk\}\}$ to blocks of size $k$.
\end{itemize}
Note that steps 1) and 2) are sufficient to ensure that $\sigma\in\widetilde{\Pi}^{uk}(q)$. This gives much flexibility to choose the remaining blocks in step 3). Since all blocks contain at least two elements, we obtain $\sigma\in\widetilde{\Pi}_{\geq 2}^{uk}(q)$.

According to steps 1)--3) there are
$$
\frac{(uk)!}{u! (k!)^u} \cdot \frac{((k-2)u)!}{u! ((k-2)!)^u} \cdot\bigg(\frac{(uk)!}{u! (k!)^u} \bigg)^{q-2}
$$
possibilities to form such partitions $\sigma\in\widetilde{\Pi}_{\geq 2}^{uk}(q)$. It is easy to verify that
$$
\frac{((k-2)u)!}{u! ((k-2)!)^u} \geq  \frac{((k-2)u)!}{u! (k!)^u} \geq \frac{(uk)!}{u! (k!)^u (2u)! \binom{uk}{(k-2)u}} \geq \frac{(uk)!}{u! (k!)^u (2u)! 2^{uk}}\,,
$$
whence
$$
|\widetilde{\Pi}_{\geq 2}^{uk}(q)| \geq \frac{((uk)!)^q}{(k!)^{uq} (u!)^q (2u)! 2^{uk}} \geq \frac{((uk)!)^{\tilde{\gamma}}}{(2(k!)^{q/k})^{uk}} \frac{((uk)!)^{q-\tilde{\gamma}}}{(u! (2u)!)^q}.
$$
Since $(uk)!\geq ((2u)!)^{k/2}$ for all even $k\in\mathbb{N}$, we can choose an even $k_0\in\N$ such that
$$
\frac{((uk_0)!)^{q-\tilde{\gamma}}}{(u! (2u)!)^q}>1\qquad\text{for all} \qquad u\in\N\,.
$$
Consequently, we have that
$$
\limsup_{u\to\infty}\frac{|\widetilde{\Pi}_{\geq 2}^{uk_0}(q)|}{((uk_0)!)^{\tilde{\gamma}}(2(k_0!)^{q/k_0})^{-uk_0}}>1\,,
$$
which yields \eqref{eq:ProofLowerBound} and completes the proof.
\end{proof}

\begin{proof}[Proof of Theorem \ref{thm:MDP}]
We let $Y_n$ be as in the statement of the theorem and fix $m\geq 3$. Then,
\begin{align*}
\Big|\cum_m\Big({Y_n\over\sqrt{\BV Y_n}}\Big)\Big| 
&= {1\over(\BV Y_n)^{m/2}}|\cum(Y_n,\ldots,Y_n)|\\
&\leq {1\over(\BV Y_n)^{m/2}}\sum_{(i_1,\ldots,i_m)\in[k]^m} |\cum(I_{q_{i_1}}(f_n^{(i_1)}),\ldots,I_{q_{i_m}}(f_n^{(i_m)}))|\,,
\end{align*}
where we used the definition of $\cum_m(\,\cdot\,)$ in terms of a joint cumulant and then applied the multilinearity of the latter together with the triangle inequality. 
Next, we apply the cumulant formula in Theorem \ref{thm:CumulantFormulaMultipleIntegrals} to deduce that
\begin{align*}
|\cum({I_{q_{i_1}}(f_n^{(i_1)})},\ldots,{I_{q_{i_m}}(f_n^{(i_m)}))}| &\leq |\widetilde{\Pi}_{\geq 2}(q_{i_1},\ldots,q_{i_m})|\\
&\qquad\times\sup_{\sigma\in\widetilde{\Pi}_{\geq 2}(q_{i_1},\ldots,q_{i_m})}\Big|\int_{\BX^{|\sigma|}}(\otimes_{\ell=1}^mf_n^{(i_\ell)})_\sigma\,\dint\mu^{|\sigma|}\Big|\,.
\end{align*}
Recalling $q=\max\{q_1,\ldots,q_k\}$, we see that according to {the} first part of Proposition \ref{Prop:tildePi}, we have
\begin{equation}\label{eqn:Bound_widetilde_Pi}
|\widetilde{\Pi}_{\geq 2}(q_{i_1},\ldots,q_{i_m})| \leq |\Pi^m(q)| \leq {q^{qm}} (m!)^q\,,
\end{equation}
while the assumption \eqref{eq:ConditionIntegrals} of the theorem ensures that
\begin{align*}
\sup_{\sigma\in\widetilde{\Pi}_{\geq 2}(q_{i_1},\ldots,q_{i_m})} {1\over(\BV Y_n)^{m/2}}\Big|\int_{\BX^{|\sigma|}}(\otimes_{\ell=1}^mf_n^{(i_\ell)})_\sigma\,\dint\mu^{|\sigma|}\Big| \leq \alpha_n^{m-2}\,.
\end{align*}
Putting the previous estimates together leads to the bound
\begin{align*}
{ \Big|\cum_m\Big({Y_n\over\sqrt{\BV Y_n}}\Big)\Big| \leq {q^{qm}} (m!)^q\alpha_n^{m-2}\,k^m 
\leq (m!)^q\,( q^{3q}k^3 \alpha_n )^{m-2}} 
\end{align*}
{for all $n\in\mathbb{N}$ and $m\geq 3$. Thus, the statements of the theorem follow from Proposition \ref{prop:CumulantBoundImplyMDPandConcentration}.} 
\end{proof}

\begin{proof}[Proof of Theorem \ref{thm:MDP2}]
The argument is very similar to the one for Theorem \ref{thm:MDP}. Basically, the only change is that now one uses the cumulant formula in Theorem \ref{thm:CumulantFormulaUStat} instead of the one in Theorem \ref{thm:CumulantFormulaMultipleIntegrals} and assumption \eqref{eq:ConditionUstatistics} instead of \eqref{eq:ConditionIntegrals}. We leave the details to the reader.
\end{proof}

\begin{remark}\rm 
Note that the second part of Proposition \ref{Prop:tildePi} implies that one cannot achieve a smaller exponent than $q$ at $m!$ in \eqref{eqn:Bound_widetilde_Pi}. This means that the choice $\gamma=q-1$ in the cumulant bound \eqref{eq:CumumantBound} cannot be improved systematically. Thus, Theorem \ref{thm:MDP} is the best result one can obtain by the method of cumulants. This agrees with the discussion next to Theorem \ref{thm:CharlierPolynomials}, which shows that Theorem \ref{thm:MDP} yields for the situation considered in Theorem \ref{thm:CharlierPolynomials} up to subpolynomial factors the optimal range of scales for the MDP. 

Also the monograph \cite{SaulisBuch} discusses cumulant bounds for multiple stochastic integrals with respect to compensated Poisson processes in Chapter 5.2. However, it appears that the estimate for the number of the involved partitions there is not correct. As discussed above the exponent at $m!$ cannot be smaller than $q$, while the exponent $q/2$ was used in \cite{SaulisBuch}.
\end{remark}

\begin{proof}[Proof of Theorem \ref{thm:CharlierPolynomials}]
We start by noting that $\BE H_q(Z_1)=0$ and $\BV H_q(Z_1)=q!$, and hence $\BE S_n=0$ and $\BV S_n=nq!$. So, according to \cite[Theorem 2.11]{Arcones} (see also \cite[Theorem 2.2]{EichelsbacherLoewe}) the sequence of random variables {$(S_n/(a_n\sqrt{nq!}))_{n\in\N}$} satisfies a MDP with speed $a_n^2$ and a good rate function {$\mathcal{I}$} with $\mathcal{I}(z)>0$ for $z\neq 0$ and $\mathcal{I}(z)\to\infty$, as $z\to\pm\infty$, if and only if 
\begin{align}\label{eq:NessesaryConditionMDP}
\limsup_{n\to\infty}{1\over a_n^2}\Big(\log n+\log\BP\big(|H_q(Z_1)|>a_n\sqrt{nq!}\big)\Big)=-\infty.
\end{align}
In addition, if condition \eqref{eq:NessesaryConditionMDP} is satisfied, then the MDP holds with the good rate function $\mathcal{I}(z)=z^2/2$.

To analyse the probability in \eqref{eq:NessesaryConditionMDP}, we start by observing that, for sufficiently large $x$, ${1\over 2}x^q\leq |H_q(x)| \leq 2x^q$ and hence, for sufficiently large $n$, we have that
\begin{align*}
\BP\big(Z_1^q>{2}a_n\sqrt{nq!}\big)\leq \BP\big(|H_q(Z_1)|>a_n\sqrt{nq!}\big)\leq \BP\Big(Z_1^q>{1\over 2}a_n\sqrt{nq!}\Big).
\end{align*}
Moreover, for sufficiently large $m\in\N$, one has that the Poisson random variable $Z_1$ satisfies $\BP(Z_1\geq m) \leq 2\BP(Z_1 = m)$. Thus, we obtain
\begin{align*}
\BP\big(Z_1=\big\lceil\big(2a_n\sqrt{nq!}\big)^{1/q}\,\big\rceil+1\big)\leq \BP\big(|H_q(Z_1)|>a_n\sqrt{nq!}\big)\leq 2\BP\Big(Z_1=\Big\lfloor\Big({1\over 2}a_n\sqrt{nq!}\Big)^{1/q}\,\Big\rfloor\Big)
\end{align*}
for sufficiently large $n$. Next, the elementary inequality $(m/2)^{m/2}\leq m!\leq m^m$ and $\BP(Z_1 = m)={e^{-1}\over m!}$ imply that
\begin{align*}
-2m\log m\leq -m\log m-1 \leq\log \BP(Z_1 = m) \leq -\frac{m}{2} (\log m-\log 2)-1\leq -\frac{m}{4} \log m
\end{align*}
for large enough $m\in\N$. It follows that
\begin{align*}
-2\Big(\big\lceil\big(2a_n\sqrt{nq!}\big)^{1/q}\,\big\rceil+1\Big)&\log\Big(\big\lceil\big(2a_n\sqrt{nq!}\big)^{1/q}\,\big\rceil+1\Big)\\
&\leq\log \BP\big(|H_q(Z_1)|>a_n\sqrt{nq!}\big) \\
&\qquad\qquad \leq \log 2-\frac{1}{4}\Big\lfloor\Big({1\over 2}a_n\sqrt{nq!}\Big)^{1/q}\,\Big\rfloor\log\Big\lfloor\Big({1\over 2}a_n\sqrt{nq!}\Big)^{1/q}\,\Big\rfloor
\end{align*}
for large enough $n\in\N$. The left- and the right-hand side in the previous chain of inequalities asymptotically behave like constant multiples of
$$
-a_n^{1/q}n^{1/(2q)}\log(a_n\sqrt{n}),
$$
as $n\to\infty$. Clearly, this expression asymptotically dominates the summand $\log n$ in \eqref{eq:NessesaryConditionMDP} by our assumption that $a_n\to\infty$, as $n\to\infty$. This means that there exist constants $c_q,C_q\in(0,\infty)$ depending only on $q$ such that
\begin{align*}
&-c_q\limsup_{n\to\infty}a_n^{-2+1/q}n^{1/(2q)}\log(a_n\sqrt{n})\\
&\qquad\qquad\leq \limsup_{n\to\infty}{1\over a_n^2}\Big(\log n+\log\BP\big(|H_q(Z_1)|>a_n\sqrt{nq!}\big)\Big)\\
&\qquad\qquad\qquad\qquad \leq -C_q\limsup_{n\to\infty}a_n^{-2+1/q}n^{1/(2q)}\log(a_n\sqrt{n}).
\end{align*}
Especially, if the sequence $(a_n)_{n\in\N}$ satisfies the conditions in part a), both the {left-hand and the right-hand} side are $-\infty$. On the other hand, the left- and the right-hand side tend to constants under the condition of part b). So, the necessary and sufficient condition \eqref{eq:NessesaryConditionMDP} for a MDP for a sequence of independent and identically distributed random variables completes the argument.
\end{proof}

\section{Proofs: Applications}\label{sec:ProofApplications}

\begin{proof}[Proof of Corollary \ref{cor:ExFixedKernel}]
It follows from \eqref{eqn:kernels} and \eqref{eqn:VarianceUstatistic} that $\BV S_n \geq v_f t_n^{2q-1}$ (see also \cite[Theorem 5.2]{ReitznerSchulte}). Together with the assumptions on $\mu_n$ and $f$ this yields that, for all $m\in\mathbb{N}$ with $m\geq 3$ and $\sigma\in{\widetilde{\Pi}}^m(q)$,
$$
\frac{1}{(\BV S_n)^{m/2}} \bigg| \int_{\BX^{|\sigma|}} (\otimes_{\ell=1}^{m} f)_\sigma \, \dint \mu_n^{|\sigma|} \bigg| \leq \frac{t_n^{|\sigma|} \|f\|_\infty^m}{(v_ft_n^{2q-1})^{m/2}}\,.
$$
Since $|\sigma|\leq (q-1)m+1$ and $t_n\geq 1$, the right-hand side is bounded by
$$
t_n^{1-m/2} \bigg( \frac{\|f\|_\infty}{\sqrt{v_f}}\bigg)^m \leq t_n^{-(m-2)/2} \max\{(\|f\|_\infty/\sqrt{v_f})^3,\|f\|_\infty/\sqrt{v_f}\}^{m-2}\,.
$$
Now the statement follows from Theorem \ref{thm:MDP2}.
\end{proof}

\begin{proof}[Proof of Corollary \ref{cor:ExRGG}]
For a connected graph $G$ with vertices $1,\hdots,\tilde{q}$ and $\diamondsuit\in\{=,\subset\}$ we have that
\begin{align}\label{eqn:BoundSubgraph}
& f_{\diamondsuit}(x_1,\hdots,x_{\tilde{q}};G,r_n) \\
& \leq \frac{1}{\tilde{q}!} \! \sum_{\varrho\in\operatorname{Per}(\tilde{q})} \!\!\! \mathbf{1}\{ x_{\varrho(i)}\leftrightarrow x_{\varrho(j)} \text{ in $\operatorname{RGG}(\{x_1,\hdots,x_{\tilde{q}}\},r_n)$ if } i\leftrightarrow j \text{ in } G \text{ for all } i,j\in\{1,\hdots,\tilde{q}\}\} \notag
\end{align}
for $x_1,\hdots,x_{\tilde{q}}\in W$, where $\leftrightarrow$ denotes an edge between two vertices and $\operatorname{Per}(\tilde{q})$ is the set of permutations of $1,\hdots,\tilde{q}$.
In the following let $m\in\N$ {with $m\geq 3$,} $i_1,\hdots,i_m\in\{1,\hdots,k\}$ and $\sigma\in \widetilde{\Pi}(i_1,\hdots,i_m)$. Then, by applying \eqref{eqn:BoundSubgraph} to the factors in the product of functions, we have that
$$
\bigg|\int_W(\otimes_{\ell=1}^m f_{\diamondsuit_{i_\ell}}(\,\cdot\,;G_{i_\ell},r_n))_\sigma\,\dint\mu_n^{|\sigma|}\bigg|\leq t_n^{|\sigma|} \Vol(W) (\kappa_dr_n^d)^{|\sigma|-1}.
$$
Here, we have used that since $\sigma\in \widetilde{\Pi}(i_1,\hdots,i_m)$, for each choice of permutations $\varrho_1,\hdots,\varrho_m$ coming from \eqref{eqn:BoundSubgraph} the points $x_1,\hdots,x_{|\sigma|}$ form a particular connected random geometric graph with distance threshold $r_n$, where it is completely determined which points are connected by edges. Hence, successive integration leads to the bound above. Thereby, the factorials in the denominator and the numbers of permutations cancel out.

Together with our variance assumption \eqref{eqn:vRGG} we obtain that
$$
\frac{\Big|\int_W(\otimes_{\ell=1}^m a_{i_\ell} f_{\diamondsuit_{i_\ell}}(\,\cdot\,;G_{i_\ell},r_n))_\sigma\,\dint\mu_n^{|\sigma|}\Big|}{(\BV S_n)^{m/2}}\leq \frac{t_n^{|\sigma|} a^m \Vol(W)(\kappa_dr_n^d)^{|\sigma|-1}}{v^{m/2} \max\{t_n^{2q-1}(\kappa_dr_n^d)^{2q-2},t_n^{p}(\kappa_dr_n^d)^{p-1}\}^{m/2}}\,.
$$
Next we consider the cases $\kappa_d t_nr_n^d\geq 1$ and $\kappa_d t_nr_n^d < 1$ separately. If $\kappa_d t_nr_n^d\geq 1$, because of $|\sigma|\leq m(q-1)+1$ we have that
\begin{align*}
\frac{\Big|\int_W(\otimes_{\ell=1}^m a_{i_\ell} f_{\diamondsuit_{i_\ell}}(\,\cdot\,;G_{i_\ell},r_n))_\sigma\,\dint\mu_n^{|\sigma|}\Big|}{(\BV S_n)^{m/2}} & \leq \frac{t_n^{m(q-1)+1} a^m \Vol(W)(\kappa_dr_n^d)^{m(q-1)}}{v^{m/2} t_n^{mq-m/2}(\kappa_dr_n^d)^{m(q-1)}} \\
& = \frac{t_n^{1-m/2} a^m \Vol(W)}{v^{m/2}}\\
& \leq \max\{1,a^2\Vol(W)/v\}^{m-2} (a/\sqrt{vt_n})^{m-2}\,.
\end{align*}
Similarly, if $\kappa_d t_nr_n^d < 1$, we obtain that
\begin{align*}
\frac{\Big|\int_W(\otimes_{\ell=1}^m a_{i_\ell} f_{\diamondsuit_{i_\ell}}(\,\cdot\,;G_{i_\ell},r_n))_\sigma\,\dint\mu_n^{|\sigma|}\Big|}{(\BV S_n)^{m/2}} 
& \leq \frac{t_n^{p} a^m \Vol(W)(\kappa_dr_n^d)^{p-1}}{v^{m/2} t_n^{mp/2}(\kappa_dr_n^d)^{m(p-1)/2}} \\
& = \frac{a^m \Vol(W)}{v^{m/2}t_n^{(m-2)p/2}(\kappa_dr_n^d)^{(m-2)(p-1)/2} } \\
&\leq \max\{1,a^2\Vol(W)/v\}^{m-2} (a/\sqrt{v t_n^p (\kappa_dr_n^d)^{p-1}})^{m-2}\,,
\end{align*}
where we used that $|\sigma|\geq p$. Now, the result follows from Theorem \ref{thm:MDP2} with $f_n^{(i)}:=a_i f_{\diamondsuit_i}(\cdot;G_i,r_n)$ for $i\in\{1,\hdots,k\}$.
\end{proof}

\begin{proof}[Proof of Corollary \ref{cor:ExOrnsteinUhlenbeck}]
It follows from the proof of the central limit theorem for $Q(T_n)$, \cite[Theorem 7.2]{PeccatiSoleTaqquUtzet}, that $Q(T_n)$ can be re-written as a sum of a first- and a second-order Wiener-It\^o integral. More precisely, we have that
$$
Q(T_n)= I_1(f_n^{(1)}) + I_2(f_n^{(2)})
$$
with the functions $f_n^{(1)}: \R\times\R\to\R$ and $f_n^{(2)}: (\R\times\R)^2\to\R$ given by
$$
f^{(1)}_n((x,u)):=u^2 \tilde{f}_n^{(1)}(x) \quad \text{ and } \quad f^{(2)}_n((x_1,u_1),(x_2,u_2)) := u_1 u_2 \tilde{f}_n^{(2)}(x_1,x_2)\,,
$$
where
\begin{align*}
\tilde{f}_n^{(1)}(x) & := {\bf 1}\{x\in(-\infty,T_n]\}  \big(e^{2\varrho x}(1-e^{-2 \varrho T_n}){\bf 1}\{x\leq 0\} \\
& \hskip 2.5cm + e^{2\varrho x}(e^{-2\varrho x}-e^{-2\varrho T_n}){\bf 1}\{x> 0\}\big)\,,\\
\tilde{f}_n^{(2)}(x_1,x_2) & := {\bf 1}\{x_1,x_2\in(-\infty,T_n]\} \big(e^{\varrho (x_1+x_2)}(1-e^{-2 \varrho T_n}){\bf 1}\{\max\{x_1,x_2\}\leq 0\}\\
& \hskip 2.5cm + e^{\varrho (x_1+x_2)}(e^{-2\varrho \max\{x_1,x_2\}}-e^{-2\varrho T_n}){\bf 1}\{\max\{x_1,x_2\}> 0\}\big)\,.
\end{align*}
Note that $\tilde{f}_n^{(1)}$ and $\tilde{f}_n^{(2)}$ are both non-negative and satisfy the estimates
\begin{equation}\label{eqn:BoundTilfefn1}
\tilde{f}_n^{(1)}(x)\leq {\bf 1}\{x\in(-\infty,T_n]\} \min\{e^{2\varrho x},1\}
\end{equation}
and
\begin{equation}\label{eqn:BoundTilfefn2}
\tilde{f}_n^{(2)}(x_1,x_2) \leq  {\bf 1}\{x_1,x_2\in(-\infty,T_n]\}\, e^{-\varrho |x_1-x_2|} \,e^{2\varrho\min\{\max\{x_1,x_2\},0\}}\,.
\end{equation}
In what follows, we let $m\geq 3$, $i_1,\hdots,i_m\in\{1,2\}$ and $\sigma\in\widetilde{\Pi}_{\geq 2}(i_1,\hdots,i_m)$. Then, we obtain
\begin{align*}
& \bigg|\int_{((-\infty,T_n]\times \R)^{|\sigma|}} (\otimes_{\ell=1}^m f_n^{(i_\ell)})_\sigma \, \dint\mu^{|\sigma|}\bigg| \\
& \leq \int_{\R^{|\sigma|}} \prod_{j=1}^{|\sigma|} (|u_j|^{m_j}+u_j^{2m_j}) \, \nu^{|\sigma|}(\dint(u_1,\hdots,u_{|\sigma|})) \int_{(-\infty,T_n]^{|\sigma|}} (\otimes_{\ell=1}^m \tilde{f}_n^{(i_\ell)})_\sigma \, \dint\lambda^{|\sigma|}\,,
\end{align*}
where we denote by $m_j$ the cardinality of the $j$th block of $\sigma$. Since $\sum_{j=1}^{|\sigma|} m_j=\sum_{\ell=1}^m i_\ell\leq 2m$ and $|\sigma|\leq \frac{1}{2}\sum_{\ell=1}^m i_\ell\leq m$, it follows from the assumptions on $\nu$ that
$$
\int_{\R^{|\sigma|}} \prod_{j=1}^{|\sigma|} (|u_j|^{m_j}+u_j^{2m_j}) \, \nu^{|\sigma|}(\dint(u_1,\hdots,u_{|\sigma|})) \leq \prod_{j=1}^{|\sigma|} (M^{m_j}+M^{2m_j}) \leq 2^m M^{4m}\,.
$$
For $j\in\{1,\hdots,|\sigma|\}$ the estimates \eqref{eqn:BoundTilfefn1} and \eqref{eqn:BoundTilfefn2} lead to
\begin{align*}
& \int_{(-\infty,T_n]^{|\sigma|}} {\bf 1}\{x_i\leq x_j, i\in\{1,\hdots,|\sigma|\}\} (\otimes_{\ell=1}^m \tilde{f}_n^{(i_\ell)})_\sigma(x_1,\hdots,x_{|\sigma|}) \, \lambda^{|\sigma|}(\dint(x_1,\hdots,x_{|\sigma|}))\\
& \leq \bigg( 2 \int_0^{\infty} e^{-\varrho s} \, \dint s \bigg)^{|\sigma|-1} \int_{-\infty}^{T_n} e^{\min\{2\varrho t,0\}} \, \dint t = (\varrho/2)^{1-|\sigma|} (T_n+1/(2\varrho))\,.
\end{align*}
Moreover, according to \eqref{eq:VarianceIntegrals} we have that
$$
\BV Q(T_n)=\|f_n^{(1)}\|_{L^2(\mu^1)}^2+ 2 \|f_n^{(2)}\|_{L^2(\mu^2)}^2\geq \|f_n^{(1)}\|_{L^2(\mu^1)}^2\,,
$$
which leads to the lower variance bound
\begin{align*}
\BV Q(T_n) & \geq \int_{(-\infty,T_n]\times\R} f_n^{(1)}(x,u)^2 \, \mu(\dint(x,u)) = \int_{-\infty}^\infty u^4 \, \nu(\dint u) \int_0^{T_n} \big(1-e^{-2\varrho (T_n-x)}\big)^2 \, \dint x \\
& \geq \int_{-\infty}^\infty u^4 \, \nu(\dint u) \int_0^{T_n} 1-2e^{-2\varrho (T_n-x)} \, \dint x \\
&= c_\nu \big(T_n-(1-e^{-2\varrho T_n})/\varrho \big) \geq c_\nu \big(T_n-1/\varrho \big)\,.
\end{align*}
Note that $1\leq |\sigma|\leq m+1$. Altogether we see that
\begin{align*}
& \frac{1}{(\BV Q(T_n))^{m/2}} \bigg|\int_{((-\infty,T_n]\times \R)^{|\sigma|}} (\otimes_{\ell=1}^m f_n^{(i_\ell)})_\sigma \, \dint\mu^{|\sigma|}\bigg|\\
& \leq (2M^4)^m
(\varrho/2)^{1-|\sigma|} |\sigma| \frac{T_n+1/(2\varrho)}{(c_\nu\big(T_n-1/\varrho \big))^{m/2}} \allowdisplaybreaks\\
& \leq (2M^4)^m
\max\{1,2/\varrho\}^m 2^m \frac{T_n+1/(2\varrho)}{(c_\nu\big(T_n-1/\varrho \big))^{m/2}} \allowdisplaybreaks\\
& \leq (4M^4)^m \max\{1,2/\varrho\}^m  \frac{T_n+1/(2\varrho)}{(c_\nu\big(T_n-1/\varrho \big))^{m/2}} \allowdisplaybreaks\\
& \leq (4M^4)^2 \max\{1,2/\varrho\}^2  \frac{T_n+1/(2\varrho)}{c_\nu\big(T_n-1/\varrho \big)} \bigg( \frac{4M^4\max\{1,2/\varrho\}}{\sqrt{c_\nu\big(T_n-1/\varrho \big)}}\bigg)^{m-2}\\
& \leq \max\bigg\{1,(4M^4)^2\max\{1,2/\varrho\}^2  \frac{T_n+1/(2\varrho)}{c_\nu\big(T_n-1/\varrho \big)}\bigg\}^{m-2} \Bigg( \frac{4M^4\max\{1,2/\varrho\}}{\sqrt{c_\nu\big(T_n-1/\varrho \big)}}\Bigg)^{m-2}\,.
\end{align*}
Thus, Theorem \ref{thm:MDP} can be applied to complete the proof.
\end{proof}

\subsection*{Acknowledgement}
CT was supported by the DFG priority program SPP 2265 \textit{Random Geometric Systems}.

\end{document}